\documentclass[a4paper]{article}
\usepackage[a4paper,margin=1in]{geometry}
\usepackage[utf8]{inputenc}
\usepackage{amsmath,amsfonts,amssymb,amsthm,bbm,mathrsfs}
\usepackage[shortlabels]{enumitem}
\usepackage{xcolor}
\usepackage{graphicx}
\usepackage{appendix}
\usepackage{setspace}
\usepackage{hyperref}
\hypersetup{
    colorlinks=true,
    linkcolor=blue,
    filecolor=magenta,      
    urlcolor=blue,
    pdftitle={Multi-objec. and hierarch. Contr. for CSPS},
    citecolor=blue,
    pdfpagemode=FullScreen,
}

\newtheorem{thm}{Theorem}[section]

\newtheorem{rmk}{Remark}[section]

\newtheorem{prop}{Proposition}[section]
\newtheorem{lm}{Lemma}[section]

\numberwithin{equation}{section}

\makeatletter
\def\@setauthors{%
  \begingroup
  \def\thanks{\protect\thanks@warning}%
  \trivlist
  \centering\footnotesize \@topsep30\p@\relax
  \advance\@topsep by -\baselineskip
  \item\relax
  \author@andify\authors
  \def\\{\protect\linebreak}%
  \authors%
  \ifx\@empty\contribs
  \else
    ,\penalty-3 \space \@setcontribs
    \@closetoccontribs
  \fi
  \endtrivlist
  \endgroup
}
\makeatother

\title{\textbf{Multi-objective and  hierarchical  control for coupled stochastic parabolic systems}}

\author{Abdellatif Elgrou$^\dagger$ \;\;and\;\; Omar Oukdach$^\ddagger$}
\date{}

\begin{document}

\maketitle
\vspace{-2em}

\begin{center}
\small
$^\dagger $Cadi Ayyad University, Faculty of Sciences Semlalia, Marrakesh, Morocco. \\
Email: \texttt{abdoelgrou@gmail.com} \\
\vspace{0.4cm}
$^\ddagger $Moulay Ismaïl University of Meknes, FST Errachidia, PMA Laboratory, BP 50, Boutalamine, Errachidia, Morocco. \\
Email: \texttt{omar.oukdach@gmail.com}
\end{center}

\vspace{1em}

\begin{abstract}
We study the Stackelberg-Nash null controllability of a coupled system governed by two linear forward stochastic parabolic equations. The system includes one leader control localized in a subset of the domain, two additional leader controls in the diffusion terms, and \( m \) follower controls, where \( m \geq 2 \). We consider two different scenarios for the followers: first, when the followers minimize a functional involving both components of the system's state, and second, when they minimize a functional involving only the second component of the state. For fixed leader controls, we first establish the existence and uniqueness of the Nash equilibrium in both scenarios and provide its characterization. As a byproduct, the problem is reformulated as a classical null controllability issue for the associated coupled forward-backward stochastic parabolic system. To address this, we derive new Carleman estimates for the adjoint stochastic systems. As far as we know, this problem is among the first to be discussed for stochastic coupled systems.
\end{abstract}
\vspace{1em}
\noindent\textbf{Keywords:} Null controllability; Stackelberg-Nash strategies; Coupled stochastic parabolic systems; Carleman estimates.\\\\
\noindent\textbf{AMS Mathematics Subject Classification:} 93B05, 93B07, 93E20, 35Q91.
\section{Introduction}

In classical single-objective optimal control problems, the goal is typically to find a control that steers the initial state to a fixed target while minimizing a specific performance criterion. However, in many practical scenarios, the use of multiple controls becomes necessary, leading to the consideration of multiple objectives that must be achieved simultaneously. This gives rise to multi-objective control problems. When a hierarchy among the objectives must be respected, the problem is referred to as a hierarchical control problem. In general, the objectives tend to be in conflict, and the optimality of one objective does not necessarily imply the optimality of others. Consequently, it is often impossible to find a strategy that satisfies all constraints simultaneously. Therefore, various compromises and trade-offs are required to attain a desirable balance among the conflicting objectives.

Depending on the nature of the problem, several approaches can be used to solve multi-objective control problems, such as Nash, Pareto, or Stackelberg strategies. These approaches are grounded in game theory, with each strategy reflecting a distinct philosophical framework. For a detailed discussion, we refer to \cite{Na51} for the noncooperative optimization strategy proposed by Nash, \cite{Pa96} for the Pareto cooperative strategy, and \cite{St34} for the Stackelberg hierarchical-cooperative strategy.

The hierarchical control problem has been extensively studied in the context of partial differential equations (PDEs) in numerous research articles. We do not attempt to provide an exhaustive list of references; instead, we aim to highlight a selection of key papers relevant to a class of PDEs. This type of problem was first introduced by J.-L. Lions for the heat and wave equations in \cite{LiPa} and \cite{LiHy}, respectively. Later, in \cite{D02} and \cite{DL04}, the authors combined the Nash and Stackelberg strategies within the context of approximate controllability. For further applications of these strategies to various evolution equations, we refer to \cite{carrSantos} for the Kuramoto–Sivashinsky equation, \cite{CF18, GRP02} for linear and semilinear parabolic equations with some numerical computations, \cite{GMR13} for Stokes equations, and \cite{ArCa24silva,GRP01} for the Burgers equation. In the context of exact null controllability, additional results can be found, for example, in \cite{ArFeGu17}, \cite{ArCaSa15}, and \cite{Calsavara} for linear and semilinear parabolic equations, \cite{AAF}, \cite{DjKDZ23} for degenerate parabolic equations, and \cite{BoMaOuNash}, \cite{BoMaOuNash2} for parabolic equations with dynamic boundary conditions.

As far as we know, \cite{oukBouElgMan}, \cite{stakNahSPEs}, and \cite{StNashdeg} are the only papers addressing Stackelberg-Nash null controllability for certain stochastic parabolic equations. The first paper considers stochastic parabolic equations with dynamic boundary conditions, the second studies  stochastic parabolic equations with Dirichlet boundary conditions, and the third investigates degenerate stochastic parabolic equations.

The objective of this study is to analyze a multi-objective control problem for coupled stochastic parabolic systems. The existing literature contains numerous classical controllability results for coupled parabolic systems, both in the deterministic case (see, e.g., \cite{surveyAmmarKBGT}, \cite{GonBurTer}) and in the stochastic case (see, e.g., \cite{Preprielgr24}, \cite{Aelgrou24CSPS}, \cite{lu2021mathematical}, and the references therein). The hierarchical control problem for deterministic coupled parabolic systems has been studied in several works; for example, we refer to \cite{sAlbSantos}, \cite{HSP18}, and \cite{HSP16}.

In this paper, we address a hierarchical control problem (using the Stackelberg-Nash strategy) for coupled stochastic parabolic systems with Dirichlet boundary conditions. The control setup includes one leader control localized in a subset of the domain, two additional leader controls in the diffusion terms, and \( m \) followers, where \( m \geq 2 \). To the best of our knowledge, this is the first study to tackle this control problem for coupled stochastic parabolic systems. For these reasons, we consider that stochastic parabolic equations present a particularly interesting field of investigation for this type of problem.  

To achieve our controllability results, the main tool is Carleman estimates, which are energy inequalities with exponential weights. These estimates were first introduced by Carleman in 1939 (see \cite{Carl39}) to prove the uniqueness of solutions for certain PDEs. Nowadays, Carleman estimates are crucial for studying various problems in applied mathematics, including control and inverse problems. Since we will use these estimates for certain coupled backward-forward stochastic parabolic equations, we refer the reader to \cite{liu2014global}, \cite[Chapter 9]{lu2021mathematical}, and \cite{tang2009null} for the derivation of some Carleman estimates for forward and backward stochastic parabolic equations.

\subsection{Statement of the problem}

Let \( T > 0 \), \( G \subset \mathbb{R}^N \) (with \( N \geq 1 \)) be an open bounded domain with a smooth boundary \( \Gamma \), and let \( m \geq 2 \) be an integer. Let \( G_0 \), \( G_i \), and \( \mathcal{O}_{i,d} \) (for \( i = 1, 2, \ldots, m \)) be nonempty open subsets of \( G \). We denote by \( \chi_{\mathcal{S}} \) the characteristic function of a subset \( \mathcal{S} \subset G \). The norm and inner product of a given Hilbert space \( \mathcal{H} \) will be denoted by \( |\cdot|_\mathcal{H} \) and \( \langle \cdot, \cdot \rangle_\mathcal{H} \), respectively. We define the following domains:
\[
Q = (0,T) \times G, \quad \Sigma = (0,T) \times \Gamma, \quad \text{and} \quad Q_0 = (0,T) \times G_0.
\]

Let \( (\Omega, \mathcal{F}, \{\mathcal{F}_t\}_{t \in [0,T]}, \mathbb{P}) \) be a fixed complete filtered probability space on which a one-dimensional standard Brownian motion \( W(\cdot) \) is defined. Here, \( \{\mathcal{F}_t\}_{t \in [0,T]} \) is the natural filtration generated by \( W(\cdot) \) and augmented by all \( \mathbb{P} \)-null sets in \( \mathcal{F} \). For a Banach space \( \mathcal{X} \), we define the following spaces equipped with their canonical norms:

\begin{itemize}
    \item \( C([0,T]; \mathcal{X}) \): the Banach space of all \( \mathcal{X} \)-valued continuous functions defined on \( [0,T] \).
    \item \( L^2_{\mathcal{F}_t}(\Omega; \mathcal{X}) \): the Banach space of all \( \mathcal{X} \)-valued \( \mathcal{F}_t \)-measurable random variables \( X \) such that
    \[
    \mathbb{E}\left[ |X|_{\mathcal{X}}^2 \right] < \infty.
    \]
    \item \( L^2_{\mathcal{F}}(0,T; \mathcal{X}) \): the Banach space consisting of all \( \mathcal{X} \)-valued \( \{\mathcal{F}_t\}_{t \in [0,T]} \)-adapted processes \( X(\cdot) \) such that
    \[
    \mathbb{E}\left[ |X(\cdot)|_{L^2(0,T; \mathcal{X})}^2 \right] < \infty.
    \]
    \item \( L^\infty_{\mathcal{F}}(0,T; \mathcal{X}) \): the Banach space consisting of all \( \mathcal{X} \)-valued \( \{\mathcal{F}_t\}_{t \in [0,T]} \)-adapted essentially bounded processes, with its usual norm denoted by \( |\cdot|_\infty \).
    \item \( L^2_{\mathcal{F}}(\Omega; C([0,T]; \mathcal{X})) \): the Banach space consisting of all \( \mathcal{X} \)-valued \( \{\mathcal{F}_t\}_{t \in [0,T]} \)-adapted continuous processes \( X(\cdot) \) such that
    \[
    \mathbb{E}\left[\max_{t \in [0,T]} |X(t)|_\mathcal{X}^2 \right] < \infty.
    \]
\end{itemize}

In what follows, we use the following notations:
\[
\mathscr{H}_i = L^2_\mathcal{F}(0,T;L^2(G_i)), \quad 
\mathscr{H}_{i,d} = L^2_\mathcal{F}(0,T;L^2(\mathcal{O}_{i,d};\mathbb{R}^2)), \quad \text{for }\, i = 1, 2, \ldots, m,
\]
and
\[
\mathscr{U} = L^2_\mathcal{F}(0,T;L^2(G_0)) \times L^2_\mathcal{F}(0,T;L^2(G)) \times L^2_\mathcal{F}(0,T;L^2(G)), \quad 
\mathscr{H} = \mathscr{H}_1 \times \cdots \times \mathscr{H}_m.
\]

We consider the following coupled forward stochastic parabolic system
\begin{equation}\label{eqq1.1}
\begin{cases}
\begin{aligned}
&dy_1 - \Delta y_1 \,dt
= \left[a_{11}y_1 + a_{12}y_2 
+ u_1\chi_{G_0} + \sum_{i=1}^m v_i\chi_{G_i}\right] \,dt + u_2 \,dW(t) \quad &\text{in } Q, \\
&dy_2 - \Delta y_2 \,dt 
= \left[a_{21}y_1 + a_{22}y_2 \right] \,dt + u_3 \,dW(t) \quad &\text{in } Q, \\
&y_1=y_2 = 0 \quad &\text{on } \Sigma, \\
&y_1(0) = y_1^0,\;\;y_2(0) = y_2^0  \quad &\text{in } G,
\end{aligned}
\end{cases}
\end{equation}
where  \( y_1^0,y_2^0 \in L^2_{\mathcal{F}_0}(\Omega; L^2(G)) \) are the initial states, and the coefficients \( a_{ij} \in L^\infty_\mathcal{F}(0,T; L^\infty(G)) \) for \( 1 \leq i,j \leq 2 \). The leader control is the triple \( (u_1, u_2, u_3) \in \mathscr{U} \), where \( u_1 \) is localized in \( G_0 \), while \( u_2 \) and \( u_3 \) are applied globally across \( G \). The follower controls \( (v_i)_{1 \leq i \leq m} \in \mathscr{H} \) are localized in \( G_i \subset G \), with the condition \( G_0 \cap G_i = \emptyset \) for \( i = 1, 2, \ldots, m \), i.e., no follower has access to change the strategy of the leader control inside \( G_0 \). The set \( G_0 \) represents the primary control domain, while \( G_i \) (\( i = 1, 2, \ldots, m \)) are the secondary control domains, all assumed to be small. Throughout this paper, \( C \) denotes a positive constant, which may vary from one place to another.

The system \eqref{eqq1.1} can be written as the following abstract forward system
\begin{equation}\label{asteqq1.1}
\begin{cases}
\begin{array}{ll}
dy - \Delta y \,dt = \left[Ay + D\left(u_1 \chi_{G_0} + \sum_{i=1}^m v_i \chi_{G_i}\right)\right] \,dt + u \,dW(t) & \text{in } Q, \\
y = 0 & \text{on } \Sigma, \\
y(0) = y^0 & \text{in } G,
\end{array}
\end{cases}
\end{equation}
where \( y = (y_1, y_2)^T \) is the state variable,  
\( y^0 = (y^0_1, y^0_2) \in L^2_{\mathcal{F}_0}(\Omega; L^2(G; \mathbb{R}^2)) \) is the initial state, 
\( D = (1, 0)^T \), \( u = (u_2, u_3)^T \), and \( A = (a_{ij})_{1 \leq i, j \leq 2} \in L^\infty_\mathcal{F}(0, T; L^\infty(G; \mathbb{R}^{2 \times 2})) \) is the coupling matrix.

From \cite{krylov},  the equation \eqref{asteqq1.1} (hence \eqref{eqq1.1}) is well-posed, i.e., for any initial condition
\( y^0 = (y^0_1, y^0_2) \in L^2_{\mathcal{F}_0}(\Omega; L^2(G; \mathbb{R}^2)) \), leader controls \( (u_1, u_2, u_3) \in \mathscr{U} \), and follower controls \( (v_i)_{1 \leq i \leq m} \in \mathscr{H} \), there exists a unique weak solution
$$
y = (y_1, y_2) \in L^2_\mathcal{F}(\Omega; C([0,T]; L^2(G; \mathbb{R}^2))) \bigcap L^2_\mathcal{F}(0,T; H^1_0(G; \mathbb{R}^2)).
$$
Moreover, there exists a positive constant \( C \) such that
\begin{align*}
& \left| y \right|_{L^2_\mathcal{F}(\Omega; C([0,T]; L^2(G; \mathbb{R}^2)))} + \left| y \right|_{L^2_\mathcal{F}(0,T; H^1_0(G; \mathbb{R}^2))} \\
& \leq C \Big( \left| y^0 \right|_{L^2_{\mathcal{F}_0}(\Omega; L^2(G; \mathbb{R}^2))} + | (u_1, u_2, u_3) |_{\mathscr{U}} + | (v_1, \dots, v_m) |_{\mathscr{H}} \Big).
\end{align*}

In this paper, we study the Stackelberg-Nash null controllability problem for the coupled linear stochastic parabolic system \eqref{eqq1.1}, which involves three leaders and \( m \) followers. We begin by formulating the problem under consideration: For fixed \( y^i_d = (y^i_{1,d}, y^i_{2,d}) \in \mathscr{H}_{i,d} \), \( i = 1, 2, \dots, m \), representing the target functions, we define two types of functionals:
\begin{enumerate}
\item The main functional $J$ is given by
\begin{equation*}
J(u_1, u_2, u_3) = \frac{1}{2} \mathbb{E} \iint_Q \left( \chi_{G_0}|u_1|^2 + |u_2|^2 + |u_3|^2 \right) \,dx\,dt.
\end{equation*}
\item The secondary functionals \( J_i \) and \( \widetilde{J}_i \) are defined as follows: For each \( i = 1, 2, \dots, m \),
\begin{align}\label{functji1}
\begin{aligned}
J_i(u_1, u_2, u_3; v_1,\dots,v_m) = & \, \frac{\alpha_i}{2} \mathbb{E} \iint_{(0,T) \times \mathcal{O}_{i,d}} \left( |y_1 - y^i_{1,d}|^2 + |y_2 - y^i_{2,d}|^2 \right) \,dx\,dt \\
& + \frac{\beta_i}{2} \mathbb{E} \iint_{(0,T) \times G_i} \rho_*^2(t)|v_i|^2 \,dx\,dt, 
\end{aligned}
\end{align}
\begin{align}\label{functji1sec}
\widetilde{J}_i(u_1, u_2, u_3; v_1,\dots,v_m) =  \frac{\alpha_i}{2} \mathbb{E} \iint_{(0,T) \times \mathcal{O}_{i,d}} |y_2 - y^i_{2,d}|^2  \,dx\,dt + \frac{\beta_i}{2} \mathbb{E} \iint_{(0,T) \times G_i} |v_i|^2 \,dx\,dt,
\end{align}
\end{enumerate}
where \( \alpha_i > 0 \), \( \beta_i \geq 1 \) are constants, and \( (y_1, y_2) \) is the solution of the system \eqref{eqq1.1}. The function \( \rho_* \) is the penalizing weight function defined in \eqref{weightfustfic}, blows up as \( t \rightarrow 0 \) and \( t \rightarrow T \). This singular behavior leads the leader controls to encounter no obstacle in achieving the null controllability of the system \eqref{eqq1.1}.

For a fixed \( (u_1, u_2, u_3) \in \mathscr{U} \), the \( m \)-tuple \( (v^*_i)_{1 \leq i \leq m} \in \mathscr{H} \) is called a Nash equilibrium for \( (J_i)_{1 \leq i \leq m} \) (resp.\ \( (\widetilde{J}_i)_{1 \leq i \leq m} \)) associated with the leader control \( (u_1, u_2, u_3) \) if, for any \( i = 1, 2, \dots, m \),
\[
J_i(u_1, u_2, u_3; v^*_1, \dots, v^*_m) = \min_{v \in \mathscr{H}_i} J_i(u_1, u_2, u_3; v^*_1, \dots, v^*_{i-1}, v, v^*_{i+1}, \dots, v^*_m),
\]
\[
\Big(\textnormal{resp.}\,\ \widetilde{J}_i(u_1, u_2, u_3; v^*_1, \dots, v^*_m) = \min_{v \in \mathscr{H}_i} \widetilde{J}_i(u_1, u_2, u_3; v^*_1, \dots, v^*_{i-1}, v, v^*_{i+1}, \dots, v^*_m)\Big).
\]
Since \( (J_i)_{1 \leq i \leq m} \) (resp.\ \( (\widetilde{J}_i)_{1 \leq i \leq m} \)) are differentiable and convex, the \( m \)-tuple \( (v^*_i)_{1 \leq i \leq m} \in \mathscr{H} \) is a Nash equilibrium for \( (J_i)_{1 \leq i \leq m} \) (resp.\ \( (\widetilde{J}_i)_{1 \leq i \leq m} \)) if and only if
\begin{equation}\label{NE7}
\left\langle \frac{\partial J_i}{\partial v_i}(u_1, u_2, u_3; v^*_1, \dots, v^*_m), v_i \right\rangle_{\mathscr{H}_i} = 0 \quad \text{for any} \quad v_i \in \mathscr{H}_i, \quad i = 1, 2, \dots, m,
\end{equation}
\begin{equation}\label{NE7sec}
\Big(\text{resp.\,\ }\left\langle \frac{\partial \widetilde{J}_i}{\partial v_i}(u_1, u_2, u_3; v^*_1, \dots, v^*_m), v_i \right\rangle_{\mathscr{H}_i} = 0 \quad \text{for any} \quad v_i \in \mathscr{H}_i, \quad i = 1, 2, \dots, m \Big).
\end{equation}
More specifically, our aim is to prove that for any initial state \( y^0=(y^0_1,y^0_2) \in L^2_{\mathcal{F}_0}(\Omega; L^2(G; \mathbb{R}^2)) \), there exists a control triple \( (u_1, u_2, u_3) \in \mathscr{U} \) that minimizes the functional \( J \), and an associated Nash equilibrium \( (v^*_i(u_1, u_2, u_3))_{1 \leq i \leq m} \in \mathscr{H} \), such that the corresponding solution \( y = (y_1, y_2) \) of the system \eqref{eqq1.1} satisfies that
\begin{equation}\label{ncontrol}
y(T, \cdot) = 0 \quad \text{in} \; G, \quad \mathbb{P}\text{-a.s.}
\end{equation}
To achieve this, we adopt the Stackelberg-Nash strategy: For fixed leader controls \( (u_1, u_2, u_3) \in \mathscr{U} \), we first determine the \( m \)-tuple Nash equilibrium for the functionals \( (J_i)_{1 \leq i \leq m} \) (resp. \( (\widetilde{J}_i)_{1 \leq i \leq m} \)). This involves finding the follower controls \( v^*_i(u_1, u_2, u_3) \in \mathscr{H}_i \) (\( i = 1, 2, \dots, m \)) that satisfy the optimality condition \eqref{NE7} (resp. \eqref{NE7sec}). Once the Nash equilibrium has been identified and fixed for each \( (u_1, u_2, u_3) \), we then proceed to determine the control triple \( (\widehat{u}_1, \widehat{u}_2, \widehat{u}_3) \in \mathscr{U} \) such that
\begin{equation*}
J(\widehat{u}_1, \widehat{u}_2, \widehat{u}_3) = \min_{(u_1, u_2, u_3) \in \mathscr{U}} J(u_1, u_2, u_3),
\end{equation*}
and the associated solution of \eqref{eqq1.1} satisfies the null controllability property \eqref{ncontrol}.

\subsection{The main results}

In this paper, we assume the following two assumptions:
\begin{enumerate}[(i)]
\item The control region \( G_0 \) and the observation domains \( \mathcal{O}_{i,d} \) satisfy
\begin{equation}\label{Assump10}
G_0 \cap \mathcal{O}_d \neq \emptyset \quad \text{and} \quad \mathcal{O}_d = \mathcal{O}_{i,d}, \quad i = 1, 2, \dots, m.
\end{equation}

\item The coupling coefficient \( a_{21} \) satisfies
\begin{align}\label{assump1.3}
    a_{21} \geq a_0 > 0 \quad \text{or} \quad -a_{21} \geq a_0 > 0, 
    \quad \text{in } (0, T) \times \mathcal{O}_0,
\end{align}
for an open set $\mathcal{O}_0\subset G_0 \cap \mathcal{O}_d$.
\end{enumerate}

The first main result of this paper is stated as follows.

\begin{thm}\label{th4.1SN}
Assume that the assumptions \eqref{Assump10} and \eqref{assump1.3} hold, and that \( \beta_i \geq1 \) for \( i = 1, 2, \dots, m \) are sufficiently large. Then, for every target functions \( (y^i_{1,d}, y^i_{2,d}) \in \mathscr{H}_{i,d} \) ($i = 1, 2, \dots, m$), and for any initial state \( (y^0_1, y^0_2) \in L^2_{\mathcal{F}_0}(\Omega; L^2(G; \mathbb{R}^2)) \), there exist controls \( (\widehat{u}_1, \widehat{u}_2, \widehat{u}_3) \in \mathscr{U} \) minimizing the functional \( J \), and an associated Nash equilibrium \( (v^*_i)_{1 \leq i \leq m} \in \mathscr{H} \) for the functionals \( (J_i)_{1 \leq i \leq m} \), such that the associated solution \( (\widehat{y}_1, \widehat{y}_2) \) of the system \eqref{eqq1.1} satisfies the controllability property
\[
\widehat{y}_1(T,\cdot) =\widehat{y}_2(T,\cdot) = 0 \quad \textnormal{in} \; G, \quad \mathbb{P}\textnormal{-a.s.}
\]
Moreover, there exists a constant $C>0$ such that
\begin{align*}
\begin{aligned}
&\, |\widehat{u}_1|^2_{L^2_\mathcal{F}(0,T; L^2(G_0))} + |\widehat{u}_2|^2_{L^2_\mathcal{F}(0,T; L^2(G))} + |\widehat{u}_3|^2_{L^2_\mathcal{F}(0,T; L^2(G))} \\
& \leq C \bigg( \mathbb{E} |y^0_1|^2_{L^2(G)} + \mathbb{E} |y^0_2|^2_{L^2(G)} + \sum_{i=1}^m \sum_{j=1}^2 \mathbb{E} \iint_{(0,T) \times \mathcal{O}_d}  |y^i_{j,d}|^2 \,dx\,dt \bigg).
\end{aligned}
\end{align*}
\end{thm}

The second main result  is presented as follows.
\begin{thm}\label{th4.1SNsec}
Assume that the assumptions \eqref{Assump10} and \eqref{assump1.3} hold, and that \( \beta_i \geq1 \) for \( i = 1, 2, \dots, m \) are sufficiently large. Then there exists a positive weight function \( \rho = \rho(t) \), which blows up as \( t \to T \), such that for every target functions \( (y^i_{1,d}, y^i_{2,d}) \in \mathscr{H}_{i,d} \) satisfying
\begin{equation}\label{asspfortargfst}
\mathbb{E} \iint_{(0, T) \times \mathcal{O}_d} \rho^2 |y^i_{j,d}|^2 \,dx\,dt < \infty, \quad i = 1, 2, \dots, m, \quad j = 1, 2,
\end{equation}
and for any initial condition \( (y^0_1, y^0_2)\in L^2_{\mathcal{F}_0}(\Omega; L^2(G; \mathbb{R}^2)) \), there exist controls \( (\widehat{u}_1, \widehat{u}_2, \widehat{u}_3) \in \mathscr{U} \) minimizing the functional \( J \), and an associated Nash equilibrium \( (v^*_i)_{1 \leq i \leq m} \in \mathscr{H} \) for the functionals \( (\widetilde{J}_i)_{1 \leq i \leq m} \), such that the associated solution \( (\widehat{y}_1, \widehat{y}_2) \) of the system \eqref{eqq1.1} satisfies that
\[
\widehat{y}_1(T,\cdot) =\widehat{y}_2(T,\cdot) = 0 \quad \textnormal{in} \; G, \quad \mathbb{P}\textnormal{-a.s.}
\]
Furthermore, there exists a constant $C>0$ such that
\begin{align*}
\begin{aligned}
&\, |\widehat{u}_1|^2_{L^2_\mathcal{F}(0,T; L^2(G_0))} + |\widehat{u}_2|^2_{L^2_\mathcal{F}(0,T; L^2(G))} + |\widehat{u}_3|^2_{L^2_\mathcal{F}(0,T; L^2(G))} \\
& \leq C \bigg( \mathbb{E} |y^0_1|^2_{L^2(G)} + \mathbb{E} |y^0_2|^2_{L^2(G)} + \sum_{i=1}^m \sum_{j=1}^2 \mathbb{E} \iint_{(0,T) \times \mathcal{O}_d}  \rho^2|y^i_{j,d}|^2 \,dx\,dt \bigg).
\end{aligned}
\end{align*}
\end{thm}

Some remarks are in order.

\begin{rmk}
The Stackelberg-Nash controllability problem for the system \eqref{eqq1.1}, where only the leader control \( u_1 \) is present or where the leader controls \( u_2 \) and \( u_3 \) are localized in specific subsets of the domain \( G \), remains an open problem. In these cases, establishing the appropriate observability inequalities seems to be a challenging problem.
\end{rmk}

\begin{rmk}
From Section \ref{section3} on Carleman estimates, it can be seen that Theorems \ref{th4.1SN} and \ref{th4.1SNsec} still hold when considering the system \eqref{eqq1.1} with more general diffusion terms involving both the states \( y_1 \) and \( y_2 \), specifically:
$$
[b_{11}y_1 + b_{12}y_2 + u_2] \, dW(t)
\quad \textnormal{and} \quad
[b_{21}y_1 + b_{22}y_2 + u_3] \, dW(t),
$$
where \( b_{ij} \in L^\infty_\mathcal{F}(0,T; L^\infty(G)) \). These terms are not introduced in  \eqref{eqq1.1} for the sake of simplicity of notations and to focus on the essential structure of the Stackelberg-Nash controllability problem.
\end{rmk}

\begin{rmk}
The condition \eqref{asspfortargfst} appears natural, as it implies that the target functions \( y^i_{j,d} \) tend to zero as \( t \to T \). This condition reflects the idea that the leader controls should not face any obstruction when attempting to control the system at time $t=T$. It remains an open question whether this condition is indeed necessary, even in the scalar deterministic and stochastic cases (see, e.g., \cite{ArCaSa15,stakNahSPEs}).
\end{rmk}

\begin{rmk}
The geometric condition \( O_d = O_{i,d} \) (\( i = 1, 2, \dots, m \)) in the assumption \eqref{Assump10} is essential for studying the Stackelberg-Nash controllability for the system \eqref{eqq1.1}. More precisely, this condition is used to derive our Carleman estimates in Section \ref{section3}. At present, we do not know whether these estimates remain valid for the case \( O_{i,d} \neq O_{j,d} \) for some \( i, j \in \{1, 2, \dots, m\} \), where \( i \neq j \). In the deterministic case, Carleman estimates with appropriate weight functions can be used to consider different configurations of the observation sets \( O_{i,d} \) (see \cite{ArFeGu17}). Therefore, it is of interest to investigate the Stackelberg-Nash controllability problem for \eqref{eqq1.1} with, at least, such alternative configurations of the sets \( O_{i,d} \).
\end{rmk}

The rest of this paper is organized as follows. In the next section, we establish the existence, uniqueness, and characterization of the Nash equilibrium for the functionals \( (J_i)_{1 \leq i \leq m} \) and \( (\widetilde{J}_i)_{1 \leq i \leq m} \). Section \ref{section3} is devoted to deriving new Carleman estimates for certain coupled adjoint backward-forward systems. In Section \ref{section4}, we prove the main controllability results, as stated in Theorems \ref{th4.1SN} and \ref{th4.1SNsec}. Finally, in Section \ref{secc5}, we present some concluding remarks.

\section{Nash equilibrium}
In this section, we establish the existence, uniqueness, and characterization of the Nash equilibrium for the functionals \( (J_i)_{1 \leq i \leq m} \) (and similarly for \( (\widetilde{J}_i)_{1 \leq i \leq m} \)) in the sense of \eqref{NE7} (and \eqref{NE7sec} for the second case). For simplicity, we provide the proof for the functionals \( (J_i)_{1 \leq i \leq m} \); analogous arguments and computations can be  applied to \( (\widetilde{J}_i)_{1 \leq i \leq m} \).

\subsection{Existence and uniqueness of the Nash equilibrium for \( (J_i)_{1 \leq i \leq m} \)}

We have the following result concerning the existence and uniqueness of the Nash equilibrium for the functionals \( (J_i)_{1 \leq i \leq m} \).

\begin{prop}\label{propp4.1}
There exists a large \( \overline{\beta} \geq1 \) such that, if \( \beta_i \geq \overline{\beta} \) for \( i = 1, 2, \dots, m \), then for each \( (u_1, u_2, u_3) \in \mathscr{U} \), there exists a unique Nash equilibrium \( (v^*_i(u_1, u_2, u_3))_{1 \leq i \leq m} \in \mathscr{H} \) for the functionals \( (J_i)_{1 \leq i \leq m} \) associated with \( (u_1, u_2, u_3) \).
\end{prop}
\begin{proof}
Consider the linear bounded operators \( \Lambda_i \in \mathcal{L}(\mathscr{H}_i; L_\mathcal{F}^2(0,T; L^2(G;\mathbb{R}^2))) \) defined by
\begin{align*}
\Lambda_i(v_i) = (y_1^i, y_2^i), \qquad i = 1, 2, \dots, m,
\end{align*}
where \( (y_1^i, y_2^i) \) is the solution of the following forward equation
\begin{equation}\label{1.10}
\begin{cases}
\begin{array}{ll}
dy_1^i - \Delta y_1^i \, dt = \left[a_{11}\,y_1^i + a_{12}\,y_2^i + v_i \chi_{G_i}\right] \, dt & \text{in} \, Q, \\
dy_2^i - \Delta y_2^i \, dt = \left[a_{21}\,y_1^i + a_{22}\,y_2^i\right] \, dt & \text{in} \, Q, \\
y_1^i = y_2^i = 0 & \text{on} \, \Sigma, \\
y_1^i(0) = y_2^i(0) = 0 & \text{in} \, G.
\end{array}
\end{cases}
\end{equation}
It is easy to see that the solution \( y = (y_1, y_2) \) of the system \eqref{eqq1.1} can be written as 
\begin{align}\label{equofy}
y = \sum_{i=1}^m \Lambda_i(v_i) + q,
\end{align}
where \( q = (q_1,q_2) \) is the solution of the forward equation
\begin{equation*}\label{1.132}
\begin{cases}
\begin{array}{ll}
dq_1 - \Delta q_1 \, dt = \left[a_{11}\,q_1 + a_{12}\,q_2 + u_1 \chi_{G_0}\right] \, dt +  u_2\,dW(t) & \text{in} \, Q, \\
dq_2 - \Delta q_2 \, dt = \left[a_{21}\,q_1 + a_{22}\,q_2\right] \, dt + u_3\,dW(t) & \text{in} \, Q, \\
q_1 =q_2= 0 & \text{on} \, \Sigma, \\
q_1(0) = y_1^0,\;\;q_2(0) = y_2^0 & \text{in} \, G.
\end{array}
\end{cases}
\end{equation*}
Let \( (u_1, u_2, u_3) \in \mathscr{U} \). Then, it is not difficult to see that for any \( v_i \in \mathscr{H}_i \),
\begin{align*}
\left\langle \frac{\partial J_i}{\partial v_i}(u_1, u_2, u_3; v^*_1, \dots, v^*_m), v_i \right\rangle_{\mathscr{H}_i} &= \alpha_i \left\langle \sum_{j=1}^m \Lambda_j(v^*_j) + q - y^i_d, \Lambda_i(v_i) \right\rangle_{\mathscr{H}_{i,d}} + \beta_i \left\langle \rho_*^2v^*_i, v_i \right\rangle_{\mathscr{H}_i},
\end{align*}
for every \( i = 1, 2, \dots, m \). Thus, \( (v^*_1, \dots, v^*_m) \) is a Nash equilibrium for \( (J_i)_{1 \leq i \leq m} \) if and only if
\begin{align}\label{4.3nashcara}
\alpha_i \left\langle \sum_{j=1}^m \Lambda_j(v^*_j) + q - y^i_d, \Lambda_i(v_i) \right\rangle_{\mathscr{H}_{i,d}} + \beta_i \left\langle \rho_*^2v^*_i, v_i \right\rangle_{\mathscr{H}_i}=0, \quad \forall v_i \in \mathscr{H}_i.
\end{align}
Equivalently, we have
\begin{align*}
\alpha_i \left\langle \Lambda_i^*\left[\sum_{j=1}^m \Lambda_j(v^*_j)\chi_{\mathcal{O}_{i,d}} - (y^i_d - q)\chi_{\mathcal{O}_{i,d}}\right], v_i \right\rangle_{\mathscr{H}_i} + \beta_i \left\langle \rho_*^2v^*_i, v_i \right\rangle_{\mathscr{H}_i} = 0, \quad \forall v_i \in \mathscr{H}_i.
\end{align*}
That is
\[
\alpha_i \Lambda_i^*\left[ \sum_{j=1}^m \Lambda_j(v^*_j)\chi_{\mathcal{O}_{i,d}} \right] + \beta_i \rho_*^2v^*_i = \alpha_i \Lambda_i^*((y^i_d - q)\chi_{\mathcal{O}_{i,d}}) \quad \text{in} \; \mathscr{H}_i, \quad i = 1, 2, \dots, m,
\]
where \( \Lambda_i^* \in \mathcal{L}(L_\mathcal{F}^2(0,T; L^2(G;\mathbb{R}^2)); \mathscr{H}_i) \) is the adjoint operator of \( \Lambda_i \). The problem is now reduced to proving that there exists a unique \( (v_1^*, \dots, v_m^*) \in \mathscr{H} \) such that
\begin{align}\label{oureqfornash}
\mathcal{L}(v_1^*, \dots, v_m^*) = (\alpha_1 \Lambda_1^*((y^1_d - q)\chi_{\mathcal{O}_{1,d}}), \dots, \alpha_m \Lambda_m^*((y^m_d - q)\chi_{\mathcal{O}_{m,d}})),
\end{align}
where \( \mathcal{L} : \mathscr{H} \to \mathscr{H} \) is the bounded operator defined by
\[
\mathcal{L}(v_1, \dots, v_m) = \left(\alpha_1 \Lambda_1^*\left[\sum_{j=1}^m \Lambda_j(v_j)\chi_{\mathcal{O}_{1,d}}\right] + \beta_1 \rho_*^2v_1, \dots, \alpha_m \Lambda_m^*\left[\sum_{j=1}^m \Lambda_j(v_j)\chi_{\mathcal{O}_{m,d}}\right] + \beta_m \rho_*^2v_m\right).
\]
It is easy to see that for any \( (v_1, \dots, v_m) \in \mathscr{H} \), we have
\[
\langle \mathcal{L}(v_1, \dots, v_m), (v_1, \dots, v_m) \rangle_\mathscr{H} = \sum_{i=1}^m \beta_i |\rho_*v_i|_{\mathscr{H}_i}^2 + \sum_{i,j=1}^m \alpha_i \langle \Lambda_j(v_j), \Lambda_i(v_i) \rangle_{\mathscr{H}_{i,d}}.
\]
This implies that
\begin{align}\label{estimonL}
\langle \mathcal{L}(v_1, \dots, v_m), (v_1, \dots, v_m) \rangle_\mathscr{H} \geq \sum_{i=1}^m (\beta_i\rho_0^2 - C) |v_i|_{\mathscr{H}_i}^2,
\end{align}
where \( \rho_0 = \min_{t \in [0,T]} \rho_*(t) \). By choosing a sufficiently large \( \overline{\beta} \geq 1 \) in \eqref{estimonL}, one can conclude that for any \( \beta_i \geq \overline{\beta} \), \( i = 1, 2, \dots, m \), we have that
\begin{align}\label{estimonLsecond}
\langle \mathcal{L}(v_1, \dots, v_m), (v_1, \dots, v_m) \rangle_\mathscr{H} \geq C |(v_1, \dots, v_m)|_\mathscr{H}^2.
\end{align}
Let us now introduce the bilinear continuous  functional \( \textbf{b}: \mathscr{H} \times \mathscr{H} \to \mathbb{R} \) as follows:
\[
\textbf{b}((v_1, \dots, v_m), (\widetilde{v}_1, \dots, \widetilde{v}_m)) = \left\langle \mathcal{L}(v_1, \dots, v_m), (\widetilde{v}_1, \dots, \widetilde{v}_m) \right\rangle_\mathscr{H},
\]
for any \( \big((v_1, \dots, v_m), (\widetilde{v}_1, \dots, \widetilde{v}_m)\big) \in \mathscr{H} \times \mathscr{H} \). We also define the linear functional \( \Psi: \mathscr{H} \to \mathbb{R} \) by
\[
\Psi(v_1, \dots, v_m) = \left\langle (v_1, \dots, v_m), \left( \alpha_1 \Lambda_1^*((y^1_d - q)\chi_{\mathcal{O}_{1,d}}), \dots, \alpha_m \Lambda_m^*((y^m_d - q)\chi_{\mathcal{O}_{m,d}}) \right) \right\rangle_\mathscr{H},
\]
for any \( (v_1, \dots, v_m) \in \mathscr{H} \). From \eqref{estimonLsecond}, it is easy to see that the bilinear functional \( \textbf{b} \) is coercive. Then, by Lax-Milgram theorem, there exists a unique \( (v^*_1, \dots, v^*_m) \in \mathscr{H} \) such that
\[
\textbf{b}((v^*_1, \dots, v^*_m), (v_1, \dots, v_m)) = \Psi(v_1, \dots, v_m) \quad \text{for any} \quad (v_1, \dots, v_m) \in \mathscr{H}.
\]
Thus, the obtained \( (v^*_1, \dots, v^*_m) \) satisfies \eqref{oureqfornash}, and therefore it is the desired Nash equilibrium for the functionals \( (J_i)_{1 \leq i \leq m} \). This completes the proof of Proposition \ref{propp4.1}.
\end{proof}
\subsection{Characterization of the Nash equilibrium for \( (J_i)_{1 \leq i \leq m} \)}
To characterize the Nash equilibrium \( (v^*_1, \dots, v^*_m) \) derived in the previous subsection, we introduce the following backward adjoint parabolic system
\begin{equation}\label{backadj}
\begin{cases} 
\begin{array}{ll}
dz_1^{i} + \Delta z_1^i \, dt = \left[-a_{11} z_1^i - a_{21} z_2^i  - \alpha_i(y_1 - y^i_{1,d}) \chi_{\mathcal{O}_{i,d}} \right] \, dt + Z_1^i \, dW(t) & \text{in} \, Q, \\
dz_2^{i} + \Delta z_2^i \, dt = \left[-a_{12} z_1^i - a_{22} z_2^i  - \alpha_i (y_2 - y^i_{2,d}) \chi_{\mathcal{O}_{i,d}} \right] \, dt + Z_2^i \, dW(t) & \text{in} \, Q, \\
z_k^i = 0, \qquad i = 1, 2, \dots, m, \quad k = 1, 2, & \text{on} \, \Sigma, \\
z_k^i(T) = 0, \qquad i = 1, 2, \dots, m, \quad k = 1, 2, & \text{in} \, G.
\end{array}
\end{cases}
\end{equation}
Applying Itô's formula for the systems \eqref{1.10} and \eqref{backadj}, we obtain that
\begin{equation}\label{equa3.6}
\alpha_i \left\langle y - y^i_d, \Lambda_i(v_i) \right\rangle_{\mathscr{H}_{i,d}} = \left\langle z_1^i, v_i \right\rangle_{\mathscr{H}_i}, \qquad i = 1, 2, \dots, m,
\end{equation}
where $y=(y_1(u_1, u_2,u_3; v_1, \dots,v_m),y_2(u_1, u_2,u_3; v_1, \dots,v_m))$ is the solution of \eqref{eqq1.1}. By combining \eqref{equa3.6}, \eqref{4.3nashcara}, and \eqref{equofy}, we deduce that \( (v^*_1, \dots, v^*_m) \) is a Nash equilibrium if and only if the following condition holds:
\[
\left\langle z_1^i, v_i \right\rangle_{\mathscr{H}_i} + \beta_i \left\langle \rho_*^2v^*_i, v_i \right\rangle_{\mathscr{H}_i} = 0 \quad \text{for any} \quad v_i \in \mathscr{H}_i, \qquad i = 1, 2, \dots, m.
\]
This leads to the following characterization of the Nash equilibrium
\begin{equation*}
v^*_i = -\frac{1}{\beta_i}\rho_*^{-2} z_1^i|_{(0,T) \times G_i}, \qquad i = 1, 2, \dots, m.
\end{equation*}

Therefore, the Stackelberg-Nash controllability problem stated in Theorem \ref{th4.1SN} reduces to establishing the classical null controllability for the following coupled forward-backward stochastic system 
\begin{equation}\label{eqq4.7}
\begin{cases}
\begin{array}{ll}
dy_1 - \Delta y_1 \, dt = \left[a_{11} y_1 + a_{12} y_2 + u_1 \chi_{G_0} - \displaystyle\sum_{i=1}^m \frac{1}{\beta_i}\rho_*^{-2} z_1^i \chi_{G_i}\right] \, dt + u_2 \, dW(t) & \text{in} \, Q, \\
dy_2 - \Delta y_2 \, dt = \left[a_{21} y_1 + a_{22} y_2 \right] \, dt + u_3 \, dW(t) & \text{in} \, Q, \\
dz_1^i + \Delta z_1^i \, dt = \left[-a_{11} z_1^i - a_{21} z_2^i  - \alpha_i (y_1 - y^i_{1,d}) \chi_{\mathcal{O}_{i,d}}\right] \, dt + Z_1^i \, dW(t) & \text{in} \, Q, \\
dz_2^i + \Delta z_2^i \, dt = \left[-a_{12} z_1^i - a_{22} z_2^i  - \alpha_i (y_2 - y^i_{2,d}) \chi_{\mathcal{O}_{i,d}}\right] \, dt + Z_2^i \, dW(t) & \text{in} \, Q, \\
y_k = z^i_k = 0, \qquad i = 1, 2, \dots, m, \quad k = 1, 2, & \text{on} \, \Sigma, \\
y_k(0) = y_k^0, \quad z_k^i(T) = 0, \qquad i = 1, 2, \dots, m, \quad k = 1, 2, & \text{in} \, G.
\end{array}
\end{cases}
\end{equation}
Using the classical duality argument, the null controllability of \eqref{eqq4.7} is reduced to establishing an observability inequality for the following adjoint backward-forward stochastic system 
\begin{equation}\label{ADJSO1}
\begin{cases}
\begin{array}{ll}
d\phi_1 + \Delta \phi_1 \, dt = \left[-a_{11} \phi_1 - a_{21} \phi_2  + \displaystyle\sum_{i=1}^m \alpha_i \psi_1^i \chi_{\mathcal{O}_{i,d}}\right] \, dt + \Phi_1 \, dW(t) & \text{in} \, Q, \\
d\phi_2 + \Delta \phi_2 \, dt = \left[-a_{12} \phi_1 - a_{22} \phi_2  + \displaystyle\sum_{i=1}^m \alpha_i \psi_2^i \chi_{\mathcal{O}_{i,d}}\right] \, dt + \Phi_2 \, dW(t) & \text{in} \, Q, \\
d\psi_1^i - \Delta \psi_1^i \, dt = \left[a_{11} \psi_1^i + a_{12} \psi_2^i + \frac{1}{\beta_i} \rho_*^{-2} \phi_1\chi_{G_i}\right] \, dt & \text{in} \, Q, \\
d\psi_2^i - \Delta \psi_2^i \, dt = \left[a_{21} \psi_1^i + a_{22} \psi_2^i \right] \, dt  & \text{in} \, Q, \\
\phi_k = \psi_k^i = 0, \qquad i = 1, 2, \dots, m, \quad k = 1, 2, & \text{on} \, \Sigma, \\
\phi_k(T) = \phi_k^T, \quad \psi_k^i(0) = 0, \qquad i = 1, 2, \dots, m, \quad k = 1, 2, & \text{in} \, G.
\end{array}
\end{cases}
\end{equation}

\subsection{Nash equilibrium for \( (\widetilde{J}_i)_{1 \leq i \leq m} \)}

Similarly to the previous subsections, we can establish the existence and uniqueness of the Nash equilibrium \( (v^*_1, \dots, v^*_m) \) for the functionals \( (\widetilde{J}_i)_{1 \leq i \leq m} \), provided that \( \beta_i \) are sufficiently large. Furthermore, we show that the \( m \)-tuple \( (v^*_1, \dots, v^*_m) \) is a Nash equilibrium for \( (\widetilde{J}_i)_{1 \leq i \leq m} \) if and only if
\[
v^*_i = -\frac{1}{\beta_i} z_1^i \big|_{(0,T) \times G_i}, \quad i = 1, 2, \dots, m,
\]
where \( (z_1^i, z_2^i;Z_1^i, Z_2^i) \), for \( i = 1, 2, \dots, m \), is the solution to the following backward adjoint system
\begin{equation}\label{backadjsec}
\begin{cases} 
\begin{array}{ll}
dz_1^{i} + \Delta z_1^i \, dt = \left[-a_{11} z_1^i - a_{21} z_2^i \right] \, dt + Z_1^i \, dW(t) & \text{in} \, Q, \\
dz_2^{i} + \Delta z_2^i \, dt = \left[-a_{12} z_1^i - a_{22} z_2^i - \alpha_i (y_2 - y^i_{2,d}) \chi_{\mathcal{O}_{i,d}} \right] \, dt + Z_2^i \, dW(t) & \text{in} \, Q, \\
z_k^i = 0, \qquad i = 1, 2, \dots, m,\quad k = 1, 2, & \text{on} \, \Sigma, \\
z_k^i(T) = 0, \qquad i = 1, 2, \dots, m,\quad k = 1, 2, & \text{in} \, G.
\end{array}
\end{cases}
\end{equation}

Thus, the Stackelberg-Nash controllability problem presented in Theorem \ref{th4.1SNsec} is reduced to establishing the null controllability of the following coupled forward-backward system
\begin{equation}\label{eqq4.7sec}
\begin{cases}
\begin{array}{ll}
dy_1 - \Delta y_1 \, dt = \left[a_{11} y_1 + a_{12} y_2 + u_1 \chi_{G_0} - \displaystyle\sum_{i=1}^m \frac{1}{\beta_i} z_1^i \chi_{G_i}\right] \, dt + u_2 \, dW(t) & \text{in} \, Q, \\
dy_2 - \Delta y_2 \, dt = \left[a_{21} y_1 + a_{22} y_2 \right] \, dt + u_3 \, dW(t) & \text{in} \, Q, \\
dz_1^i + \Delta z_1^i \, dt = \left[-a_{11} z_1^i - a_{21} z_2^i\right] \, dt + Z_1^i \, dW(t) & \text{in} \, Q, \\
dz_2^i + \Delta z_2^i \, dt = \left[-a_{12} z_1^i - a_{22} z_2^i  - \alpha_i (y_2 - y^i_{2,d}) \chi_{\mathcal{O}_{i,d}}\right] \, dt + Z_2^i \, dW(t) & \text{in} \, Q, \\
y_k = z^i_k = 0, \quad i = 1, 2, \dots, m, \quad k = 1, 2, & \text{on} \, \Sigma, \\
y_k(0) = y_k^0, \quad z_k^i(T) = 0, \quad i = 1, 2, \dots, m, \quad k = 1, 2, & \text{in} \, G.
\end{array}
\end{cases}
\end{equation}
The null controllability of \eqref{eqq4.7sec} can be derived by proving an observability inequality for the corresponding adjoint backward-forward system
\begin{equation}\label{ADJSO1sec}
\begin{cases}
\begin{array}{ll}
d\phi_1 + \Delta \phi_1 \, dt = \left[-a_{11} \phi_1 - a_{21} \phi_2\right] \, dt + \Phi_1 \, dW(t) & \text{in} \, Q, \\
d\phi_2 + \Delta \phi_2 \, dt = \left[-a_{12} \phi_1 - a_{22} \phi_2  + \displaystyle\sum_{i=1}^m \alpha_i \psi_2^i \chi_{\mathcal{O}_{i,d}}\right] \, dt + \Phi_2 \, dW(t) & \text{in} \, Q, \\
d\psi_1^i - \Delta \psi_1^i \, dt = \left[a_{11} \psi_1^i + a_{12} \psi_2^i + \frac{1}{\beta_i}  \phi_1\chi_{G_i} \right] \, dt & \text{in} \, Q, \\
d\psi_2^i - \Delta \psi_2^i \, dt = \left[a_{21} \psi_1^i + a_{22} \psi_2^i \right] \, dt & \text{in} \, Q, \\
\phi_k = \psi_k^i = 0, \quad i = 1, 2, \dots, m, \quad k = 1, 2, & \text{on} \, \Sigma, \\
\phi_k(T) = \phi_k^T, \quad \psi_k^i(0) = 0, \quad i = 1, 2, \dots, m, \quad k = 1, 2, & \text{in} \, G.
\end{array}
\end{cases}
\end{equation}
\begin{rmk}
The systems \eqref{ADJSO1} and \eqref{ADJSO1sec} are almost identical, with the only difference being in the right-hand side of the first and third equations. In the next section, we will derive the desired Carleman estimates for these systems using similar techniques. However, the term ``\( \sum_{i=1}^m \alpha_i \psi_1^i \chi_{\mathcal{O}_{i,d}} \)'' in the first equation of \eqref{ADJSO1} complicates the derivation of the required Carleman estimate for \eqref{ADJSO1}. For this reason, the term ``\( \frac{1}{\beta_i} \rho_*^{-2} \phi_1 \chi_{G_i} \)'' in the third equation, with the function \( \rho_* \) defined in \eqref{weightfustfic}, will play a crucial role to overcome some computational difficulties in deriving the Carleman estimates. Note that, since the term ``\( \sum_{i=1}^m \alpha_i \psi_1^i \chi_{\mathcal{O}_{i,d}} \)'' does not appear in the first equation of the system \eqref{ADJSO1sec}, the needed Carleman estimate for \eqref{ADJSO1sec} is obtained without the introduction of the function \( \rho_* \).
\end{rmk}
\section{Carleman estimates for the systems \eqref{ADJSO1} and \eqref{ADJSO1sec}}\label{section3}
\subsection{Some preliminaries}
In this subsection, we introduce the necessary notations and briefly recall some well-known Carleman estimates for forward and backward stochastic parabolic equations. These estimates will be essential for deriving our Carleman estimates for the coupled systems \eqref{ADJSO1} and \eqref{ADJSO1sec}. To provide context, we first state the following result (see \cite{BFurIman}).
\begin{lm}\label{lmm5.1}
For any nonempty open subset \( \mathcal{B} \Subset G \), there exists a function \( \eta_0 \in C^4(\overline{G}) \) such that
$$
\eta_0 > 0 \;\; \textnormal{in} \; G, \quad \eta_0 = 0 \;\; \textnormal{on} \; \Gamma, \quad |\nabla \eta_0| > 0 \;\; \textnormal{in} \; \overline{G \setminus \mathcal{B}}.
$$
\end{lm}
For some parameters \( \lambda, \mu \geq 1 \) (large enough), we define the following functions:
\begin{align}\label{weightfustfic}
    \begin{aligned}
&\alpha \equiv \alpha(t,x) = \frac{e^{\mu \eta_0(x)} - e^{2\mu |\eta_0|_\infty}}{t(T - t)}, \qquad \gamma \equiv \gamma(t) = \frac{1}{t(T - t)}, \qquad\theta \equiv \theta(t,x) = e^{\lambda \alpha},\\
&\hspace{2.5cm}\rho_*\equiv\rho_*(t)=e^{-\frac{\lambda \alpha^*(t)}{2}},\qquad \alpha^*(t) = \min_{x \in \overline{G}} \alpha(t, x).
    \end{aligned}
\end{align}
It is easy to see that there exists a constant \( C = C(G,T) > 0 \) such that for all $(t,x)\in Q$,
\begin{align}\label{aligned123}
\begin{aligned}
&\gamma \geq C, \qquad |\gamma'| \leq C \gamma^2, \qquad |\gamma''| \leq C \gamma^3, \\
&|\alpha_t| \leq C e^{2\mu |\eta_0|_\infty} \gamma^2, \qquad |\alpha_{tt}| \leq C e^{2\mu |\eta_0|_\infty} \gamma^3.
\end{aligned}
\end{align}
For \( \mathcal{B} \Subset G \) a nonempty open subset, \( d \in \mathbb{R} \), and a process \( z \), we denote by
\[
I(d, z) = \lambda^d \mathbb{E} \iint_Q \theta^2 \gamma^d z^2 \,dx\,dt + \lambda^{d-2} \mathbb{E} \iint_Q \theta^2 \gamma^{d-2} |\nabla z|^2 \,dx\,dt,
\]
and
\[
\mathcal{L}_{\mathcal{B}}(d, z) = \lambda^d \mathbb{E} \int_{0}^T \int_{\mathcal{B}} \theta^2 \gamma^d z^2 \,dx\,dt.
\]
Let us first introduce the following forward stochastic parabolic equation
\begin{equation}\label{eqqgfr}
\begin{cases}
\begin{array}{ll}
dz - \Delta z \,dt = F_0 \,dt + F_1\,dW(t) & \text{in} \, Q, \\
z = 0 & \text{on} \, \Sigma, \\
z(0) = z_0 & \text{in} \, G,
\end{array}
\end{cases}
\end{equation}
where \( z_0 \in L^2_{\mathcal{F}_0}(\Omega; L^2(G)) \) and \( F_0, F_1 \in L^2_{\mathcal{F}}(0, T; L^2(G)) \). From \cite[Lemma 2.2]{Preprielgr24}, we recall the following Carleman estimate.

\begin{lm}\label{lm1.1}
Let \( \mathcal{B} \Subset G \) be a nonempty open subset and \( d \in \mathbb{R} \). There exist a large \( \mu_0 \geq 1 \) such that for \( \mu = \mu_0 \), one can find a positive constant \( C \) and a large \( \lambda_0 \) depending only on \( G \), \( \mathcal{B} \), \( \mu_0 \), \( T \), and \( d \), such that for all \( \lambda \geq \lambda_0 \), \( F_0, F_1 \in L^2_{\mathcal{F}}(0, T; L^2(G)) \), and \( z_0 \in L^2_{\mathcal{F}_0}(\Omega; L^2(G)) \), the associated solution \( z \) of the equation \eqref{eqqgfr} satisfies that 
\begin{align}\label{carfor5.6}
\begin{aligned}
I(d,z) \leq C \left[ \mathcal{L}_{\mathcal{B}}(d,z) + \lambda^{d-3} \mathbb{E} \iint_Q \theta^2 \gamma^{d-3} F_0^2 \,dx\,dt + \lambda^{d-1} \mathbb{E} \iint_Q \theta^2 \gamma^{d-1} F_1^2 \,dx\,dt \right].
\end{aligned}
\end{align}
\end{lm}
We also consider the following backward stochastic parabolic equation
\begin{equation}\label{eqqgbc}
\begin{cases}
dz + \Delta z\,dt = F\,dt + Z \,dW(t) & \text{in}\, Q, \\ 
z = 0 & \text{on}\, \Sigma, \\
z(T) = z_T & \text{in}\, G,
\end{cases}
\end{equation}
where \( z_T \in L^2_{\mathcal{F}_T}(\Omega; L^2(G)) \) and \( F \in L^2_{\mathcal{F}}(0,T; L^2(G)) \). We have the following Carleman estimate for solutions of \eqref{eqqgbc} (see \cite[Lemma 3.3]{Aelgrou24CSPS}).

\begin{lm}\label{lm1.13.22}
Let \( \mathcal{B} \Subset G \) be a nonempty open subset and \( d \in \mathbb{R} \). There exists a large \( \mu_1 \geq 1 \) such that for \( \mu = \mu_1 \), one can find a positive constant \( C \) and a large \( \lambda_1 \geq 1 \) depending only on \( G \), \( \mathcal{B} \), \( \mu_1 \), \( T \), and \( d \), such that for all \( \lambda \geq \lambda_1 \), \( F \in L^2_{\mathcal{F}}(0, T; L^2(G)) \) and \( z_T \in L^2_{\mathcal{F}_T}(\Omega; L^2(G)) \), the associated weak solution \( (z, Z) \) of the equation \eqref{eqqgbc} satisfies 
\begin{align}\label{3.22carlemgenBack}
\begin{aligned}
I(d,z) \leq C\left[\mathcal{L}_{\mathcal{B}}(d,z) + \lambda^{d-3} \mathbb{E} \iint_Q \theta^2 \gamma^{d-3} F^2 \,dx\,dt + \lambda^{d-1} \mathbb{E} \iint_Q \theta^2 \gamma^{d-1} Z^2 \,dx\,dt\right].
\end{aligned}
\end{align}
\end{lm}
In what follows, we set \( \mu = \overline{\mu} = \max(\mu_0, \mu_1) \), where \( \mu_0 \) (resp. \( \mu_1 \)) is the constant given in Lemma \ref{lm1.1} (resp. Lemma \ref{lm1.13.22}). Moreover, to derive improved Carleman estimates for the systems \eqref{ADJSO1} and \eqref{ADJSO1sec}, we will introduce the following modified weight functions:
\begin{align}\label{rec1}
\overline{\alpha} &= \overline{\alpha}(t,x) = \frac{e^{\overline{\mu}\eta_0(x)} - e^{2\overline{\mu}\vert\eta_0\vert_\infty}}{\ell(t)}, \qquad
\overline{\gamma} = \overline{\gamma}(t) = \frac{1}{\ell(t)}, \qquad
\overline{\theta} = e^{\lambda\overline{\alpha}},
\end{align}
where
\begin{equation*}
\ell(t) = 
\begin{cases}
    \begin{array}{ll}
        T^2/4 & \textnormal{for }\, 0 \leq t \leq T/2, \\
        t(T-t) & \textnormal{for }\, T/2 \leq t \leq T.
    \end{array}
\end{cases}
\end{equation*}
We also denote by
\begin{align}\label{inteIbar}
\overline{I}_{t_1,t_2}(d,z) = \mathbb{E} \int_{t_1}^{t_2} \int_G \overline{\theta}^2 \overline{\gamma}^d z^2 \,dx\,dt 
+ \mathbb{E} \int_{t_1}^{t_2} \int_G \overline{\theta}^2 \overline{\gamma}^{d-2} \vert\nabla z\vert^2 \,dx\,dt, \quad 0 \leq t_1 \leq t_2 \leq T.
\end{align}
\subsection{Carleman estimates for the systems \eqref{ADJSO1} and \eqref{ADJSO1sec}}
Applying the inequalities \eqref{carfor5.6} and \eqref{3.22carlemgenBack}, we will first prove the following Carleman estimate for the coupled system \eqref{ADJSO1}.

\begin{lm}\label{thmm5.1}
Assume that \eqref{Assump10} and \eqref{assump1.3} hold. There exists a constant \(C > 0\) and a sufficiently large \(\overline{\lambda},\overline{\beta} \geq 1\) (depending only on \(G\), \(G_0\), \(\mathcal{O}_d\), \(\overline{\mu}\), \(\alpha_i\), and \(a_{ij}\)) such that for all \(\lambda \geq \overline{\lambda}\), \(\beta_i \geq \overline{\beta}\), $i=1,2,\dots,m$, and \(\phi_k^T \in L^2_{\mathcal{F}_T}(\Omega; L^2(G))\) (\(k=1,2\)), the corresponding solution of the system \eqref{ADJSO1} satisfies that
\begin{align}\label{ineq5.1}
\begin{aligned}
&I(5,\phi_1) + I(5,\phi_2) + I(3,h_1) + I(3,h_2) \\
&\leq C\bigg[\lambda^{29} \mathbb{E} \iint_{Q_0} \theta^2 \gamma^{29} |\phi_1|^2 \,dx\,dt + \lambda^{15} \mathbb{E} \iint_Q \theta^2 \gamma^{15} |\Phi_1|^2 \,dx\,dt \\
&\hspace{1cm}   + \lambda^{9} \mathbb{E} \iint_Q \theta^2 \gamma^{9} |\Phi_2|^2 \,dx\,dt\bigg],
\end{aligned}
\end{align}
where \( h_k = \sum_{i=1}^m \alpha_i \psi_k^i \), \,\(k = 1,2\).
\end{lm}
\begin{proof} The proof is organized into several steps.\\
\textbf{Step 1.} \textbf{Preliminary notations and estimations.}\\
We first define the following nonempty open subsets \( \mathcal{O}_i \) (\(i=1,2,3,4\)) such that
\[
\mathcal{O}_4 \Subset\mathcal{O}_3 \Subset \mathcal{O}_2 \Subset \mathcal{O}_1 \Subset \mathcal{O}_0 \subset G_0 \cap \mathcal{O}_d,
\]
where \( \mathcal{O}_0 \) is the subset defined in \eqref{assump1.3}. We also consider the functions \( \xi_i \in C^{\infty}(\mathbb{R}^N) \) (for the existence of such functions, see, e.g., \cite{HSP18}) satisfying
\begin{align}\label{assmzeta}
\begin{aligned}
& 0 \leq \xi_i \leq 1, \quad \xi_i = 1 \,\, \text{in} \,\, \mathcal{O}_{5-i}, \quad \text{Supp}(\xi_i) \subset \mathcal{O}_{4-i}, \\ 
& \frac{\Delta \xi_i}{\xi_i^{1/2}} \in L^\infty(G), \quad \frac{\nabla \xi_i}{\xi_i^{1/2}} \in L^\infty(G; \mathbb{R}^N), \quad i=1,2,3,4.
\end{aligned}
\end{align}
Set $w_\nu=\theta^2(\lambda\gamma)^\nu$, it is easy to see that for a large enough $\lambda$ and any $\nu\in\mathbb{N}$, we have
\begin{align}\label{timedrrives}
|\partial_t w_\nu| \leq C \lambda^{\nu+2} \theta^2 \gamma^{\nu+2},\qquad |\nabla(w_\nu\xi_i)| \leq C \lambda^{\nu+1} \theta^2 \gamma^{\nu+1}\xi_i^{1/2}, \quad i=1,2,3,4.
\end{align}

From \eqref{ADJSO1}, noting that \( (h_1, h_2) \) solves the following forward stochastic parabolic system
\begin{equation}\label{systeofh1h2}
\begin{cases}
\begin{aligned}
& dh_1 - \Delta h_1 \, dt = \left[ a_{11} h_1+a_{12} h_2 + \sum_{i=1}^m \frac{\alpha_i}{\beta_i} \rho_*^{-2}\phi_1  \chi_{G_i} \right] \, dt  && \text{in } Q, \\
& dh_2 - \Delta h_2 \, dt = \left[ a_{21} h_1+a_{22} h_2  \right] \, dt  && \text{in } Q, \\
& h_1=h_2 = 0  && \text{on } \Sigma, \\
& h_1(0)=h_2(0) = 0 && \text{in } G.
\end{aligned}
\end{cases}
\end{equation}
Applying the Carleman estimate \eqref{carfor5.6} for solutions of the system \eqref{systeofh1h2} with \( \mathcal{B} = \mathcal{O}_4 \) and \( d = 3 \), we conclude, for a sufficiently large \( \lambda \), 
\begin{align}\label{Car4.13}
\begin{aligned}
I(3, h_1) + I(3, h_2) \leq C \bigg( \mathcal{L}_{\mathcal{O}_4}(3, h_1) + \mathcal{L}_{\mathcal{O}_4}(3, h_2) + \mathbb{E} \iint_Q \theta^2 |\phi_1|^2 \,dx\,dt \bigg).
\end{aligned}
\end{align}
Using the Carleman estimate \eqref{3.22carlemgenBack} to the equations satisfied by the state \( (\phi_1, \phi_2) \) in \eqref{ADJSO1}, with \( \mathcal{B} = \mathcal{O}_4 \) and \( d = 5 \). Then by choosing a sufficiently large \( \lambda \), we find that
\begin{align}\label{car4.14}
\begin{aligned}
&I(5, \phi_1) + I(5, \phi_2) \\
&\leq C \bigg( \mathcal{L}_{\mathcal{O}_4}(5, \phi_1) + \mathcal{L}_{\mathcal{O}_4}(5, \phi_2) + \lambda^{2} \mathbb{E} \iint_Q \theta^2 \gamma^{2} |h_1|^2 \,dx\,dt \\
&\hspace{1cm} + \lambda^{2} \mathbb{E} \iint_Q \theta^2 \gamma^{2} |h_2|^2 \,dx\,dt + \lambda^{4} \mathbb{E} \iint_Q \theta^2 \gamma^{4} |\Phi_1|^2 \,dx\,dt \\
&\hspace{1cm} + \lambda^{4} \mathbb{E} \iint_Q \theta^2 \gamma^{4} |\Phi_2|^2 \,dx\,dt \bigg).
\end{aligned}
\end{align}
Adding \eqref{Car4.13} and \eqref{car4.14} and choosing a large enough $\lambda$, we obtain that
\begin{align}\label{firsine1}
\begin{aligned}
&I(3, h_1) + I(3, h_2) + I(5, \phi_1) + I(5, \phi_2) \\
&\leq C \bigg( \mathcal{L}_{\mathcal{O}_4}(3, h_1) + \mathcal{L}_{\mathcal{O}_4}(3, h_2) + \mathcal{L}_{\mathcal{O}_4}(5, \phi_1) + \mathcal{L}_{\mathcal{O}_4}(5, \phi_2) \\
&\hspace{1cm} + \lambda^{4} \mathbb{E} \iint_Q \theta^2 \gamma^{4} |\Phi_1|^2 \,dx\,dt + \lambda^{4} \mathbb{E} \iint_Q \theta^2 \gamma^{4} |\Phi_2|^2 \,dx\,dt \bigg).
\end{aligned}
\end{align}
\textbf{Step 2.} \textbf{Absorbing the localized terms on \( h_1 \) and \( h_2 \).}\\
We first note that
\begin{align}\label{ineqforp1h1p2h2}
\mathcal{L}_{\mathcal{O}_4}(3, h_1) + \mathcal{L}_{\mathcal{O}_4}(3, h_2) \leq \mathbb{E} \iint_Q w_3 \xi_1 |h_1|^2 \,dx\,dt + \mathbb{E} \iint_Q w_3 \xi_1 |h_2|^2 \,dx\,dt.
\end{align}
Using Itô's formula for \( \sum_{j=1}^2 d(w_3 \xi_1 h_j \phi_j) \), we have that
\begin{align}\label{ineqq1}
\begin{aligned}
&\mathbb{E} \iint_Q w_3 \xi_1 |h_1|^2 \,dx\,dt + \mathbb{E} \iint_Q w_3 \xi_1 |h_2|^2 \,dx\,dt \\
&=-\mathbb{E}\iint_Q \partial_t w_3\,\xi_1 h_1\phi_1 dx \,dt-\mathbb{E}\iint_Q \partial_t w_3\,\xi_1 h_2\phi_2 dx \,dt
\\
& \quad - \mathbb{E} \iint_Q h_1\nabla\phi_1\cdot\nabla(w_3\xi_1)\,dx\,dt +\mathbb{E} \iint_Q \phi_1\nabla h_1\cdot\nabla(w_3\xi_1)\,dx\,dt
\\
& \quad  -\sum_{i=1}^m\frac{\alpha_i}{\beta_i} \mathbb{E} \iint_Q w_3\xi_1\rho_*^{-2}|\phi_1|^2 \chi_{G_i}\,dx\,dt - \mathbb{E} \iint_Q h_2\nabla\phi_2\cdot\nabla(w_3\xi_1)\,dx\,dt
\\
& \quad  +\mathbb{E} \iint_Q \phi_2\nabla h_2\cdot\nabla(w_3\xi_1)\,dx\,dt
\\
&= \sum_{1\leq i\leq7} J_i.
\end{aligned}
\end{align}
Let \( \varepsilon > 0 \). Using \eqref{timedrrives}, we observe that
\begin{align*}
    J_1\leq C \lambda^{5} \mathbb{E} \iint_Q \theta^2 \gamma^{5}\xi_1 |h_1||\phi_1|\,dx\,dt.
\end{align*}
Applying Young's inequality, it follows that
\begin{align}\label{ineqqforj1711}
    J_1\leq \varepsilon I(3, h_1)+\frac{C}{\varepsilon} \lambda^{7} \mathbb{E} \iint_Q \theta^2 \gamma^{7}\xi_1 |\phi_1|^2 \,dx\,dt.
\end{align}
Similarly, we get
\begin{align}\label{ineqqforj1712}
    J_2\leq \varepsilon I(3, h_2)+\frac{C}{\varepsilon} \lambda^{7} \mathbb{E} \iint_Q \theta^2 \gamma^{7}\xi_1 |\phi_2|^2 \,dx\,dt.
\end{align}
We also have that
\begin{align*}
    J_3\leq C \lambda^{4} \mathbb{E} \iint_Q \theta^2 \gamma^{4}\xi_1^{1/2} |h_1||\nabla\phi_1| \,dx\,dt,
\end{align*}
which leads to
\begin{align}\label{ineqqforj17123}
    J_3\leq \varepsilon I(3, h_1)+\frac{C}{\varepsilon} \lambda^{5} \mathbb{E} \iint_Q \theta^2 \gamma^{5}\xi_1 |\nabla\phi_1|^2 \,dx\,dt.
\end{align}
In a similar manner, we find that
\begin{align}\label{ineqqforj17124}
    J_4\leq \varepsilon I(3, h_1)+\frac{C}{\varepsilon} \lambda^{7} \mathbb{E} \iint_Q \theta^2 \gamma^{7}\xi_1 |\phi_1|^2 \,dx\,dt.
\end{align}
Additionally, using similar techniques as above, we obtain
\begin{align}\label{ineqqforj17134}
    J_5\leq C \lambda^{3} \mathbb{E} \iint_Q \theta^2 \gamma^{3}\xi_1 |\phi_1|^2 \,dx\,dt,
\end{align}

\begin{align}\label{ineqqforj17126}
    J_6 \leq \varepsilon I(3, h_2)+\frac{C}{\varepsilon} \lambda^{5} \mathbb{E} \iint_Q \theta^2 \gamma^{5}\xi_1 |\nabla\phi_2|^2 \,dx\,dt,
\end{align}
and
\begin{align}\label{ineqqforj17127}
    J_7\leq \varepsilon I(3, h_2)+\frac{C}{\varepsilon} \lambda^{7} \mathbb{E} \iint_Q \theta^2 \gamma^{7}\xi_1 |\phi_2|^2 \,dx\,dt.
\end{align}
Combining \eqref{firsine1}, \eqref{ineqforp1h1p2h2}, \eqref{ineqq1}, \eqref{ineqqforj1711}, \eqref{ineqqforj1712}, \eqref{ineqqforj17123}, \eqref{ineqqforj17124},
\eqref{ineqqforj17134}, \eqref{ineqqforj17126}, \eqref{ineqqforj17127} and taking a small enough \( \varepsilon \) and a large \( \lambda \), we conclude that
\begin{align}\label{estima3.177sec1}
\begin{aligned}
&I(3, h_1) + I(3, h_2) + I(5, \phi_1) + I(5, \phi_2) \\
&\leq C \bigg( \mathcal{L}_{\mathcal{O}_3}(7,\phi_1) + \mathcal{L}_{\mathcal{O}_3}(7,\phi_2)+\lambda^{5} \mathbb{E} \int_0^T \int_{\mathcal{O}_3} \theta^2 \gamma^{5} |\nabla\phi_1|^2 \,dx\,dt\\
&\hspace{1cm}+\lambda^{5} \mathbb{E} \int_0^T \int_{\mathcal{O}_3} \theta^2 \gamma^{5} |\nabla\phi_2|^2 \,dx\,dt + \lambda^{4} \mathbb{E} \iint_Q \theta^2 \gamma^{4} |\Phi_1|^2 \,dx\,dt\\
&\hspace{1cm} + \lambda^{4} \mathbb{E} \iint_Q \theta^2 \gamma^{4} |\Phi_2|^2 \,dx\,dt \bigg).
\end{aligned}
\end{align}
\textbf{Step 3.} \textbf{Absorbing the localized term on \( \nabla\phi_2 \).}\\
Notice that
\begin{align}\label{ineqq3.22ste4step3}
\lambda^{5} \mathbb{E} \int_0^T \int_{\mathcal{O}_3} \theta^2 \gamma^{5} |\nabla \phi_2|^2 \,dx\,dt \leq \mathbb{E} \iint_Q w_5 \xi_2 |\nabla \phi_2|^2 \,dx\,dt.
\end{align}
Next, we compute \( d(w_5 \xi_2 |\phi_2|^2) \). By applying integration by parts, we arrive at
\begin{align}\label{equalisecgradphi1step3}
\begin{aligned}
2 \mathbb{E} \iint_Q w_5 \xi_2 |\nabla \phi_2|^2 \,dx\,dt &= - \mathbb{E} \iint_Q \partial_t w_5 \, \xi_2 |\phi_2|^2 \,dx\,dt - 2 \mathbb{E} \iint_Q \phi_2 \nabla \phi_2\cdot\nabla(w_5 \xi_2)  \,dx\,dt \\
&\quad + 2 \mathbb{E} \iint_Q w_5 \xi_2 a_{12} \phi_1\phi_2 \,dx\,dt+ 2 \mathbb{E} \iint_Q w_5 \xi_2 a_{22} |\phi_2|^2 \,dx\,dt   \\
&\quad - 2 \mathbb{E} \iint_Q w_5 \xi_2 \phi_2 h_2 \chi_{\mathcal{O}_d} \,dx\,dt- \mathbb{E} \iint_Q w_5 \xi_2 |\Phi_2|^2 \,dx\,dt \\
&= \sum_{8 \leq i \leq 13} J_i.
\end{aligned}
\end{align}
Let \( \varepsilon > 0 \). By applying \eqref{timedrrives} and Young's inequality, it is straightforward to derive that
\begin{align}\label{ineeqK1st310}
J_{8} \leq C \lambda^{7} \mathbb{E} \iint_Q \theta^2  \gamma^{7} \xi_2 |\phi_2|^2 \,dx\,dt,
\end{align}

\begin{align}\label{ineeqK1st311}
J_{9} \leq \varepsilon I(5,\phi_2)+\frac{C}{\varepsilon}\lambda^{9} \mathbb{E} \iint_Q \theta^2  \gamma^{9} \xi_2 |\phi_2|^2 \,dx\,dt,
\end{align}

\begin{align}\label{ineeqK1st311thir}
J_{10} \leq \varepsilon I(5,\phi_2)+\frac{C}{\varepsilon}\lambda^{5} \mathbb{E} \iint_Q \theta^2  \gamma^{5} \xi_2 |\phi_1|^2 \,dx\,dt,
\end{align}

\begin{align}\label{ineeqK1st312}
J_{11} \leq C \lambda^{5} \mathbb{E} \iint_Q \theta^2  \gamma^{5} \xi_2 |\phi_2|^2 \,dx\,dt,
\end{align}

\begin{align}\label{ineeqK1st313}
J_{12} \leq \varepsilon I(3,h_2)+\frac{C}{\varepsilon} \lambda^{7} \mathbb{E} \iint_Q \theta^2  \gamma^{7} \xi_2 |\phi_2|^2 \,dx\,dt,
\end{align}
and
\begin{align}\label{ineeqK1st314}
J_{13} \leq \lambda^{5} \mathbb{E} \iint_Q \theta^2  \gamma^{5}  |\Phi_2|^2 \,dx\,dt.
\end{align}
Combining \eqref{estima3.177sec1}, \eqref{ineqq3.22ste4step3}, \eqref{equalisecgradphi1step3}, \eqref{ineeqK1st310}, \eqref{ineeqK1st311}, \eqref{ineeqK1st311thir}, \eqref{ineeqK1st312}, \eqref{ineeqK1st313}, \eqref{ineeqK1st314}, and taking a small enough $\varepsilon$ and a large $\lambda$, we deduce that
\begin{align}\label{estimsesec1nestp3}
\begin{aligned}
&I(3, h_1) + I(3, h_2) + I(5, \phi_1) + I(5, \phi_2) \\
&\leq C \bigg( \mathcal{L}_{\mathcal{O}_2}(7,\phi_1) + \mathcal{L}_{\mathcal{O}_2}(9,\phi_2) +\lambda^{5} \mathbb{E} \int_0^T \int_{\mathcal{O}_2} \theta^2 \gamma^{5} |\nabla\phi_1|^2 \,dx\,dt \\
&\hspace{1cm}+ \lambda^{4} \mathbb{E} \iint_Q \theta^2 \gamma^{4} |\Phi_1|^2 \,dx\,dt + \lambda^{5} \mathbb{E} \iint_Q \theta^2 \gamma^{5} |\Phi_2|^2 \,dx\,dt \bigg).
\end{aligned}
\end{align}
\textbf{Step 4.} \textbf{Absorbing the localized term on \( \phi_2 \).}\\
Recalling the assumption \eqref{assump1.3} and without loss of generality, we can assume that
\begin{equation*}
a_{21} \geq a_0 > 0, \quad \text{in } (0, T) \times \mathcal{O}_0.
\end{equation*}
Next, we observe that
\begin{align}\label{ineqq3.22}
    a_0 \mathcal{L}_{\mathcal{O}_2}(9, \phi_2) \leq \mathbb{E} \iint_Q w_9 \xi_3 a_{21} |\phi_2|^2 \,dx\,dt.
\end{align}
Using Itô's formula for \( d(w_9 \xi_3 \phi_1 \phi_2) \), we arrive at
\begin{align}\label{3.24estprin}
    \begin{aligned}
    \mathbb{E} \iint_Q w_9 \xi_3 a_{21} |\phi_2|^2 \,dx\,dt &= \mathbb{E} \iint_Q \partial_t w_9 \, \xi_3 \phi_1 \phi_2 \,dx\,dt + \mathbb{E} \iint_Q \nabla \phi_2 \cdot \nabla (w_9 \xi_3 \phi_1) \,dx\,dt \\
    &\quad - \mathbb{E} \iint_Q w_9 \xi_3 a_{12} |\phi_1|^2 \,dx\,dt- \mathbb{E} \iint_Q w_9 \xi_3 a_{22} \phi_1 \phi_2 \,dx\,dt\\
    &\quad  + \mathbb{E} \iint_Q w_9 \xi_3 \phi_1 h_2 \chi_{\mathcal{O}_d} \,dx\,dt  + \mathbb{E} \iint_Q \nabla \phi_1 \cdot \nabla (w_9 \xi_3 \phi_2) \,dx\,dt\\
    &\quad  - \mathbb{E} \iint_Q w_9 \xi_3 a_{11} \phi_1 \phi_2 \,dx\,dt + \mathbb{E} \iint_Q w_9 \xi_3 \phi_2 h_1 \chi_{\mathcal{O}_d} \,dx\,dt \\
    &\quad + \mathbb{E} \iint_Q w_9 \xi_3 \Phi_1 \Phi_2 \,dx\,dt \\
    &= \sum_{14 \leq i \leq 22} J_i.
    \end{aligned}
\end{align}
We first see that
\[
J_{14} \leq C \lambda^{11} \mathbb{E} \iint_Q \theta^2  \gamma^{11}\xi_3 |\phi_1| |\phi_2| \,dx\,dt.
\]
Let \( \varepsilon > 0 \). By Young's inequality, it follows that
\begin{align}\label{ineeJ1}
J_{14} \leq \varepsilon I(5, \phi_2) + \frac{C}{\varepsilon} \lambda^{17} \mathbb{E} \iint_Q \theta^2 \gamma^{17} \xi _3|\phi_1|^2 \,dx\,dt.
\end{align}
It is not difficult to see that
\begin{align*}
J_{15} \leq C \lambda^{10} \mathbb{E} \iint_Q \theta^2 \gamma^{10}\xi_3^{1/2} |\nabla \phi_2||\phi_1| \,dx\,dt + C \lambda^{9} \mathbb{E} \iint_Q \theta^2 \gamma^{9} \xi_3 |\nabla \phi_1| |\nabla \phi_2| \,dx\,dt,
\end{align*}
which leads to
\begin{align}\label{ineqj5}
\begin{aligned}
J_{15} \leq &\ \varepsilon I(5, \phi_2) + \frac{C}{\varepsilon} \lambda^{17} \mathbb{E} \iint_Q \theta^2  \gamma^{17} \xi_3 |\phi_1|^2 \,dx\,dt \\
& + \frac{C}{\varepsilon} \lambda^{15} \mathbb{E} \iint_Q \theta^2  \gamma^{15} \xi_3 |\nabla \phi_1|^2 \,dx\,dt.
\end{aligned}
\end{align}
Notice that
\begin{align}\label{ineqj3third}
J_{16} \leq C \lambda^{9} \mathbb{E} \iint_Q \theta^2  \gamma^{9} \xi_3|\phi_1|^2 \,dx\,dt.
\end{align}
It is easy to see that
\begin{align}\label{ineqj3}
J_{17} \leq \varepsilon I(5, \phi_2) + \frac{C}{\varepsilon} \lambda^{13} \mathbb{E} \iint_Q \theta^2  \gamma^{13} \xi_3|\phi_1|^2 \,dx\,dt,
\end{align}
\begin{align}\label{ineqj4}
J_{18} \leq \varepsilon I(3, h_2) + \frac{C}{\varepsilon} \lambda^{15} \mathbb{E} \iint_Q \theta^2  \gamma^{15} \xi_3|\phi_1|^2 \,dx\,dt.
\end{align}
We also have that
\begin{align*}
J_{19} \leq C \lambda^{10} \mathbb{E} \iint_Q \theta^2 \gamma^{10} \xi_3^{1/2} |\nabla \phi_1| |\phi_2| \,dx\,dt + C \lambda^{9} \mathbb{E} \iint_Q \theta^2 \gamma^{9} \xi_3 |\nabla \phi_1| |\nabla \phi_2| \,dx\,dt,
\end{align*}
which leads to
\begin{align}\label{ineqj2}
\begin{aligned}
J_{19} \leq &\ \varepsilon I(5, \phi_2) + \frac{C}{\varepsilon} \lambda^{15} \mathbb{E} \iint_Q \theta^2 \gamma^{15} \xi_3 |\nabla \phi_1|^2 \,dx\,dt.
\end{aligned}
\end{align}
It is straightforward to see that
\begin{equation}\label{ineqj6}
J_{20} \leq \varepsilon I(5, \phi_2) + \frac{C}{\varepsilon} \lambda^{13} \mathbb{E} \iint_Q \theta^2  \gamma^{13} \xi_3|\phi_1|^2 \,dx\,dt,
\end{equation}
\begin{equation}\label{ineqj7}
J_{21} \leq \varepsilon I(5, \phi_2) + \frac{C}{\varepsilon} \lambda^{13} \mathbb{E} \iint_Q \theta^2 \gamma^{13} \xi_3|h_1|^2 \,dx\,dt,
\end{equation}
and
\begin{equation}\label{ineqj12}
J_{22} \leq \lambda^{9} \mathbb{E} \iint_Q \theta^2 \gamma^{9} |\Phi_1|^2 \,dx\,dt + \lambda^{9} \mathbb{E} \iint_Q \theta^2 \gamma^{9} |\Phi_2|^2 \,dx\,dt.
\end{equation}
Combining \eqref{estimsesec1nestp3}, \eqref{ineqq3.22}, \eqref{3.24estprin}, \eqref{ineeJ1}, \eqref{ineqj5}, \eqref{ineqj3third}, \eqref{ineqj3}, \eqref{ineqj4}, \eqref{ineqj2}, \eqref{ineqj6}, \eqref{ineqj7}, \eqref{ineqj12}, taking a small enough \( \varepsilon \), and a large \( \lambda \), we conclude that
\begin{align}\label{estim3.37secc}
\begin{aligned}
&I(3, h_1) + I(3, h_2) + I(5, \phi_1) + I(5, \phi_2) \\
&\leq C \bigg( \mathcal{L}_{\mathcal{O}_1}(17,\phi_1) + \mathcal{L}_{\mathcal{O}_1}(13,h_1) + \lambda^{15} \mathbb{E} \int_0^T \int_{\mathcal{O}_1} \theta^2 \gamma^{15} |\nabla \phi_1|^2 \,dx\,dt \\
&\hspace{1cm} + \lambda^{9} \mathbb{E} \iint_Q \theta^2 \gamma^{9} |\Phi_1|^2 \,dx\,dt  + \lambda^{9} \mathbb{E} \iint_Q \theta^2 \gamma^{9} |\Phi_2|^2 \,dx\,dt \bigg).
\end{aligned}
\end{align}
\textbf{Step 5.} \textbf{Absorbing the localized term on $\nabla\phi_1$.}\\
We first note that
\begin{align}\label{ineqq3.22ste4}
\lambda^{15} \mathbb{E} \int_0^T \int_{\mathcal{O}_1} \theta^2 \gamma^{15} |\nabla \phi_1|^2 \,dx\,dt \leq \mathbb{E} \iint_Q w_{15} \xi_4 |\nabla \phi_1|^2 \,dx\,dt.
\end{align}
Next, we compute \( d(w_{15} \xi_4 |\phi_1|^2) \). By applying integration by parts, we arrive at
\begin{align}\label{equalisecgradphi1}
\begin{aligned}
2 \mathbb{E} \iint_Q w_{15} \xi_4 |\nabla \phi_1|^2 \,dx\,dt &= - \mathbb{E} \iint_Q \partial_t w_{15} \,\xi_4 |\phi_1|^2 \,dx\,dt - 2 \mathbb{E} \iint_Q \phi_1 \nabla \phi_1\cdot\nabla(w_{15} \xi_4)  \,dx\,dt \\
&\quad + 2 \mathbb{E} \iint_Q w_{15} \xi_4 a_{11} |\phi_1|^2 \,dx\,dt + 2 \mathbb{E} \iint_Q w_{15} \xi_4 a_{21} \phi_1 \phi_2 \,dx\,dt \\
&\quad - 2 \mathbb{E} \iint_Q w_{15} \xi_4 \phi_1 h_1 \chi_{\mathcal{O}_d} \,dx\,dt - \mathbb{E} \iint_Q w_{15} \xi_4 |\Phi_1|^2 \,dx\,dt \\
&= \sum_{23 \leq i \leq 28} J_i.
\end{aligned}
\end{align}
Let \( \varepsilon > 0 \). We begin by noting that
\begin{align}\label{ineeqK1}
J_{23} \leq C \lambda^{17} \mathbb{E} \iint_Q \theta^2  \gamma^{17} \xi_4 |\phi_1|^2 \,dx\,dt.
\end{align}
Additionally, we have that
\[
J_{24} \leq C \lambda^{16} \mathbb{E} \iint_Q \theta^2  \gamma^{16} \xi_4^{1/2} |\phi_1| |\nabla \phi_1| \,dx\,dt,
\]
which leads to
\begin{align}\label{ineeqK2}
J_{24} \leq \varepsilon I(5, \phi_1) 
+ \frac{C}{\varepsilon} \lambda^{29} \mathbb{E} \iint_Q \theta^2 \gamma^{29} \xi_4 |\phi_1|^2 \,dx\,dt.
\end{align}
It is straightforward to see that
\begin{align}\label{ineeqK3}
J_{25} \leq C \lambda^{15} \mathbb{E} \iint_Q \theta^2 \gamma^{15} \xi_4 |\phi_1|^2 \,dx\,dt,
\end{align}
\begin{align}\label{ineeqK712}
J_{26} \leq \varepsilon I(5, \phi_2) 
+ \frac{C}{\varepsilon} \lambda^{25} \mathbb{E} \iint_Q \theta^2 \gamma^{25} \xi_4 |\phi_1|^2 \,dx\,dt.
\end{align}
We also have that
\begin{align}\label{ineeqK7}
J_{27} \leq \varepsilon I(3, h_1) 
+ \frac{C}{\varepsilon} \lambda^{27} \mathbb{E} \iint_Q \theta^2 \gamma^{27}\xi_4 |\phi_1|^2 \,dx\,dt,
\end{align}
and
\begin{align}\label{ineeqK8}
J_{28} \leq \lambda^{15} \mathbb{E} \iint_Q \theta^2 \gamma^{15} |\Phi_1|^2 \,dx\,dt.
\end{align}
Combining  \eqref{estim3.37secc}, \eqref{ineqq3.22ste4}, \eqref{equalisecgradphi1}, \eqref{ineeqK1}, \eqref{ineeqK2}, \eqref{ineeqK3}, \eqref{ineeqK712}, \eqref{ineeqK7}, \eqref{ineeqK8}, and taking a small \( \varepsilon \) and a large \( \lambda \), we find that
\begin{align}\label{ineaastep4}
\begin{aligned}
&I(3, h_1) + I(3, h_2) + I(5, \phi_1) + I(5, \phi_2) \\
&\leq C \bigg( \mathcal{L}_{\mathcal{O}_0}(29,\phi_1)+ \mathcal{L}_{\mathcal{O}_0}(13,h_1) + \lambda^{15} \mathbb{E} \iint_Q \theta^2 \gamma^{15} |\Phi_1|^2 \,dx\,dt \\
&\hspace{1cm}   + \lambda^{9} \mathbb{E} \iint_Q \theta^2 \gamma^{9} |\Phi_2|^2 \,dx\,dt \bigg).
\end{aligned}
\end{align}
\textbf{Step 6.} \textbf{Absorbing the localized term on $h_1$.}\\
By the classical energy estimate for the system  \eqref{systeofh1h2}, we have that
\begin{align*}
    |h_1|^2_{L^2_{\mathcal{F}}(0,T;L^2(G))} + |h_2|^2_{L^2_{\mathcal{F}}(0,T;L^2(G))} \leq C \left|\sum_{i=1}^m \frac{\alpha_i}{\beta_i}\rho_*^{-2}\phi_1\chi_{G_i}\right|_{L^2_{\mathcal{F}}(0,T;L^2(G))}^2,
\end{align*}
which leads to the following estimate:
\begin{align}\label{keyineqh1}
    \mathbb{E}\iint_Q (|h_1|^2+|h_2|^2)\,dx\,dt \leq C\sum_{i=1}^m \frac{\alpha_i^2}{\beta_i^2}\mathbb{E}\iint_Q \rho_*^{-4}|\phi_1|^2 \,dx\,dt.
\end{align}
Since \( \rho_*^{-4} \leq \theta^2 \), the inequality \eqref{keyineqh1} implies that 
\begin{align}\label{thisineh1}
C\mathcal{L}_{\mathcal{O}_0}(13,h_1)\leq C\lambda^{8}\sum_{i=1}^m \frac{\alpha_i^2}{\beta_i^2} I(5,\phi_1).
\end{align}
Taking sufficiently large \( \beta_i \), for \( i = 1, 2, \dots, m \), in \eqref{thisineh1}, we obtain
\begin{align}\label{inelastoneca}
C\mathcal{L}_{\mathcal{O}_0}(13,h_1)\leq \frac{1}{2}I(5,\phi_1).
\end{align}
Therefore, combining \eqref{inelastoneca} and \eqref{ineaastep4}, we deduce that
\begin{align*}
\begin{aligned}
&I(3, h_1) + I(3, h_2) + I(5, \phi_1) + I(5, \phi_2) \\
&\leq C \bigg( \mathcal{L}_{\mathcal{O}_0}(29,\phi_1)\leq + \lambda^{15} \mathbb{E} \iint_Q \theta^2 \gamma^{15} |\Phi_1|^2 \,dx\,dt \\
&\hspace{1cm}   + \lambda^{9} \mathbb{E} \iint_Q \theta^2 \gamma^{9} |\Phi_2|^2 \,dx\,dt \bigg),
\end{aligned}
\end{align*}
which implies the desired inequality \eqref{ineq5.1}. This concludes the proof of Lemma \ref{thmm5.1}.
\end{proof}

By repeating the computations in the proof of Lemma \ref{thmm5.1}, one can obtain a Carleman estimate similar to \eqref{ineq5.1} for the system \eqref{ADJSO1sec}. However, in this case, the term ``\( h_1 = \sum_{i=1}^m \alpha_i \psi_1^i \)'' does not appear on the right-hand side of the first equation in \eqref{ADJSO1sec}, which simplifies the approach. Specifically, to derive this Carleman estimate for the system \eqref{ADJSO1sec}, we proceed as follows:
\begin{enumerate}
\item By adapting the ideas from Steps 1 to 4 in the proof of Lemma \ref{thmm5.1}, we derive the inequality \eqref{estim3.37secc} for the system \eqref{ADJSO1sec}, but without the localized term involving \( h_1 \) on the right-hand side, as the term \( J_{21} = 0 \) in the identity \eqref{3.24estprin}.
    \item Similarly to Step 5 in the proof of Lemma \ref{thmm5.1}, we absorb the localized term involving \( \nabla \phi_1 \) on the right-hand side of \eqref{estim3.37secc}.
\item Consequently, we establish the following Carleman estimate for the system \eqref{ADJSO1sec}.
\end{enumerate}
\begin{lm}\label{thmm5.2sec}
Assume that \eqref{Assump10} and \eqref{assump1.3} hold. There exists a constant \(C > 0\) and a sufficiently large \(\overline{\lambda} \geq 1\) (depending only on \(G\), \(G_0\), \(\mathcal{O}_d\), \(\overline{\mu}\), \(\alpha_i\) and \(a_{ij}\)) such that for all \(\lambda \geq \overline{\lambda}\), and \(\phi_k^T \in L^2_{\mathcal{F}_T}(\Omega; L^2(G))\) (\(k=1,2\)), the corresponding solution of the system \eqref{ADJSO1sec} satisfies that
\begin{align}\label{ineq5.1sec}
\begin{aligned}
&I(5,\phi_1) + I(5,\phi_2) + I(3,h_1) + I(3,h_2) \\
&\leq C\bigg[\lambda^{29} \mathbb{E} \iint_{Q_0} \theta^2 \gamma^{29} |\phi_1|^2 \,dx\,dt + \lambda^{15} \mathbb{E} \iint_Q \theta^2 \gamma^{15} |\Phi_1|^2 \,dx\,dt \\
&\hspace{1cm}   + \lambda^{9} \mathbb{E} \iint_Q \theta^2 \gamma^{9} |\Phi_2|^2 \,dx\,dt\bigg],
\end{aligned}
\end{align}
where \( h_k = \sum_{i=1}^m \alpha_i \psi_k^i \), \,\(k = 1,2\).
\end{lm}

\subsection{Improved Carleman estimates for the systems \eqref{ADJSO1} and \eqref{ADJSO1sec}}

In the rest of this subsection, due to the presence of source terms \(y^i_d = (y^i_{1,d}, y^i_{2,d})\), \(i = 1, 2, \dots, m\), in the system \eqref{eqq4.7}, we will require the following improved Carleman estimate for the system \eqref{ADJSO1}.
\begin{thm}\label{lem4.5st} 
Assume that \eqref{Assump10} and \eqref{assump1.3} hold. There exists a constant \(C > 0\), and a large \(\overline{\beta} \geq1\)  such that for all \(\beta_i \geq \overline{\beta}\), \((i = 1, 2, \dots, m)\),   $\lambda\geq\overline{\lambda}$ given in Lemma \ref{thmm5.1}, and \(\phi_k^T \in L^2_{\mathcal{F}_T}(\Omega; L^2(G))\) (\(k = 1, 2\)), the corresponding solution of the system \eqref{ADJSO1} satisfies that
\begin{align}\label{improvedCarl}
\begin{aligned}
& \mathbb{E} \left|\phi_1(0)\right|^2_{L^2(G)} + \mathbb{E} \left|\phi_2(0)\right|^2_{L^2(G)}  + \overline{I}_{0,T}(5,\phi_1) + \overline{I}_{0,T}(5,\phi_2)  \\
&\quad + \overline{I}_{0,T}(3,h_1) + \overline{I}_{0,T}(3,h_2) \\
&\leq C\bigg[\lambda^{29} \mathbb{E} \iint_{Q_0} \theta^2 \gamma^{29} |\phi_1|^2 \,dx\,dt + \lambda^{15} \mathbb{E} \iint_Q \theta^2 \gamma^{15} |\Phi_1|^2 \,dx\,dt \\
&\hspace{1cm}   + \lambda^{9} \mathbb{E} \iint_Q \theta^2 \gamma^{9} |\Phi_2|^2 \,dx\,dt\bigg],
\end{aligned}
\end{align}
where \( h_k = \sum_{i=1}^m \alpha_i \psi_k^i \), \;\( k = 1,2 \).
\end{thm}
\begin{proof}
Let \(\kappa \in C^1([0,T])\) be the function such that
\begin{equation}\label{kappadef}
\kappa =
\begin{cases}
1 & \text{in } [0, T/2], \\
0 & \text{in } [3T/4, T],
\end{cases}
\quad \text{and} \quad |\kappa'(t)| \leq C/T.
\end{equation}
Set \((\widetilde{\phi}_1, \widetilde{\phi}_2; \widetilde{\Phi}_1, \widetilde{\Phi}_2) = \kappa (\phi_1, \phi_2; \Phi_1, \Phi_2)\). Then \((\widetilde{\phi}_1, \widetilde{\phi}_2; \widetilde{\Phi}_1, \widetilde{\Phi}_2)\) is the solution of the system
\begin{equation}\label{equaphitilde}
\begin{cases}
\begin{array}{ll}
d\widetilde{\phi}_1 + \Delta \widetilde{\phi}_1 \, dt = \left[-a_{11} \widetilde{\phi}_1 - a_{21} \widetilde{\phi}_2 + \kappa h_1\chi_{\mathcal{O}_d} + \kappa' \phi_1\right] \, dt + \widetilde{\Phi}_1 \, dW(t) & \text{in } Q, \\
d\widetilde{\phi}_2 + \Delta \widetilde{\phi}_2 \, dt = \left[-a_{12} \widetilde{\phi}_1-a_{22} \widetilde{\phi}_2 + \kappa h_2\chi_{\mathcal{O}_d} + \kappa' \phi_2\right] \, dt + \widetilde{\Phi}_2 \, dW(t) & \text{in } Q, \\
\widetilde{\phi}_1=\widetilde{\phi}_2 = 0 & \text{on } \Sigma, \\
\widetilde{\phi}_1(T)=\widetilde{\phi}_2(T) = 0 & \text{in } G.
\end{array}
\end{cases}
\end{equation}
From \eqref{kappadef} and the classical energy estimate for \eqref{equaphitilde}, it is easy to see that
\begin{align*}
\begin{aligned}
&\, \mathbb{E} \left|\phi_1(0)\right|^2_{L^2(G)} + \mathbb{E} \left|\phi_2(0)\right|^2_{L^2(G)} 
+ \mathbb{E} \int_0^{T/2} \int_G \big(|\phi_1|^2 + |\phi_2|^2 \big) \,dx\,dt \\
&\leq C \bigg[ \mathbb{E} \int_{T/2}^{3T/4} \int_G \big(|\phi_1|^2 + |\phi_2|^2\big) \,dx\,dt 
+ \mathbb{E} \int_{0}^{3T/4} \int_G \big(|h_1|^2 + |h_2|^2\big) \,dx\,dt \bigg].
\end{aligned}
\end{align*}
Recalling \eqref{inteIbar}, it follows that
\begin{align}\label{3.46inee}
    \begin{aligned}
&\, \mathbb{E} |\phi_1(0)|^2_{L^2(G)} + \mathbb{E} |\phi_2(0)|^2_{L^2(G)} + \overline{I}_{0,\frac{T}{2}}(5,\phi_1) + \overline{I}_{0,\frac{T}{2}}(5,\phi_2) \\
&\leq C \bigg[ \overline{I}_{0,\frac{T}{2}}(3,h_1) + \overline{I}_{0,\frac{T}{2}}(3,h_2) 
+ \mathbb{E} \int_{T/2}^{3T/4} \int_G \big(|\phi_1|^2 + |\phi_2|^2 + |h_1|^2 + |h_2|^2\big) \,dx\,dt \bigg],
\end{aligned}
\end{align}
which implies that
\begin{align}\label{estimmae3.56}
    \begin{aligned}
&\, \mathbb{E} |\phi_1(0)|^2_{L^2(G)} + \mathbb{E} |\phi_2(0)|^2_{L^2(G)} + \overline{I}_{0,\frac{T}{2}}(5,\phi_1) + \overline{I}_{0,\frac{T}{2}}(5,\phi_2) + \overline{I}_{0,\frac{T}{2}}(3,h_1) + \overline{I}_{0,\frac{T}{2}}(3,h_2) \\
&\leq C \bigg[ \overline{I}_{0,\frac{T}{2}}(3,h_1) + \overline{I}_{0,\frac{T}{2}}(3,h_2) 
+ \mathbb{E} \int_{T/2}^{3T/4} \int_G \big(|\phi_1|^2 + |\phi_2|^2 + |h_1|^2 + |h_2|^2\big) \,dx\,dt \bigg].
\end{aligned}
\end{align}
Since \(\theta = \overline{\theta}\) and \(\gamma = \overline{\gamma}\) on \((T/2, T)\), it is not difficult to see that
\begin{align}\label{estim3.577}
\begin{aligned}
&\, \overline{I}_{\frac{T}{2},T}(5,\phi_1) + \overline{I}_{\frac{T}{2},T}(5,\phi_2) 
+ \overline{I}_{\frac{T}{2},T}(3,h_1) + \overline{I}_{\frac{T}{2},T}(3,h_2) \\
&\leq \left[ I(5,\phi_1) + I(5,\phi_2) + I(3,h_1) + I(3,h_2) \right].
\end{aligned}
\end{align}
Adding \eqref{estim3.577} and \eqref{estimmae3.56}, we get
\begin{align}\label{ineqq3.58new}
    \begin{aligned}
&\, \mathbb{E} |\phi_1(0)|^2_{L^2(G)} + \mathbb{E} |\phi_2(0)|^2_{L^2(G)} 
+ \overline{I}_{0,T}(5,\phi_1) + \overline{I}_{0,T}(5,\phi_2) + \overline{I}_{0,T}(3,h_1) + \overline{I}_{0,T}(3,h_2) \\
&\leq C \bigg[ \overline{I}_{0,\frac{T}{2}}(3,h_1) + \overline{I}_{0,\frac{T}{2}}(3,h_2) 
+ \mathbb{E} \int_{T/2}^{3T/4} \int_G \big(|\phi_1|^2 + |\phi_2|^2 + |h_1|^2 + |h_2|^2\big) \,dx\,dt \\
&\hspace{1cm} + I(5,\phi_1) + I(5,\phi_2) + I(3,h_1) + I(3,h_2) \bigg].
\end{aligned}
\end{align}
Notice also that
\begin{align}\label{iequuinet3.599}
\mathbb{E} \int_{T/2}^{3T/4} \int_G \big(|\phi_1|^2 + |\phi_2|^2 + |h_1|^2 + |h_2|^2\big) \,dx\,dt
\leq C \big[I(5,\phi_1) + I(5,\phi_2) + I(3,h_1) + I(3,h_2)\big].
\end{align}
From \eqref{iequuinet3.599} and \eqref{ineqq3.58new}, we get
\begin{align}\label{ineqq3.58}
    \begin{aligned}
&\, \mathbb{E} |\phi_1(0)|^2_{L^2(G)} + \mathbb{E} |\phi_2(0)|^2_{L^2(G)} 
+ \overline{I}_{0,T}(5,\phi_1) + \overline{I}_{0,T}(5,\phi_2) + \overline{I}_{0,T}(3,h_1) + \overline{I}_{0,T}(3,h_2) \\
&\leq C \bigg[ \overline{I}_{0,\frac{T}{2}}(3,h_1) + \overline{I}_{0,\frac{T}{2}}(3,h_2) 
+ I(5,\phi_1) + I(5,\phi_2) + I(3,h_1) + I(3,h_2) \bigg].
\end{aligned}
\end{align}
To eliminate the first two terms on the right-hand side of \eqref{ineqq3.58}, we use the classical energy estimate for the state \((h_1, h_2)\) in the system \eqref{ADJSO1}. It is easy to see that
\begin{align*}
    \mathbb{E} \int_0^{T/2} \int_G \big( |h_1|^2 + |h_2|^2 \big) \,dx\,dt
    \leq C \sum_{i=1}^m \frac{\alpha_i^2}{\beta_i^2} \mathbb{E} \iint_Q \rho_*^{-4} |\phi_1|^2 \,dx\,dt.
\end{align*}
Since $\rho_*^{-4}\leq\overline{\theta}^2$, we conclude that
\begin{align}\label{inequl3.59}
    \mathbb{E} \int_0^{T/2} \int_G \big( |h_1|^2 + |h_2|^2 \big) \,dx\,dt
    \leq C \sum_{i=1}^m \frac{\alpha_i^2}{\beta_i^2} \mathbb{E} \iint_Q \overline{\theta}^2 |\phi_1|^2 \,dx\,dt.
\end{align}
Recalling \eqref{rec1}-\eqref{inteIbar}, and by selecting sufficiently large \( \beta_i \geq \overline{\beta} \) for \( i = 1, 2, \dots, m \), we obtain that
\begin{align}\label{ineqaull3.62}
    C\left[\overline{I}_{0,\frac{T}{2}}(3,h_1) + \overline{I}_{0,\frac{T}{2}}(3,h_2)\right] 
    \leq \frac{1}{2} \overline{I}_{0,T}(5,\phi_1).
\end{align}
From \eqref{ineqaull3.62} and \eqref{ineqq3.58}, we conclude that
\begin{align}\label{ineqq3.58laste}
\begin{aligned}
&\, \mathbb{E} |\phi_1(0)|^2_{L^2(G)} + \mathbb{E} |\phi_2(0)|^2_{L^2(G)} + \overline{I}_{0,T}(5,\phi_1) + \overline{I}_{0,T}(5,\phi_2) \\
&\quad + \overline{I}_{0,T}(3,h_1) + \overline{I}_{0,T}(3,h_2) \\
&\leq C \big[I(5,\phi_1) + I(5,\phi_2) + I(3,h_1) + I(3,h_2)\big].
\end{aligned}
\end{align}
Finally, by combining \eqref{ineqq3.58laste} and \eqref{ineq5.1}, we establish the desired estimate \eqref{improvedCarl}.
\end{proof}

Similarly to the proof of Theorem \ref{lem4.5st} and considering the presence of the source terms \( y^i_{2,d}\), \(i = 1, 2, \dots, m\), in the system \eqref{eqq4.7sec}, we derive the following improved Carleman estimate for  \eqref{ADJSO1sec}.
\begin{thm}\label{lth5.1sec} 
Assume that \eqref{Assump10} and \eqref{assump1.3} hold. There exists a constant \(C > 0\), and a large \(\overline{\beta} \geq1\)  such that for all \(\beta_i \geq \overline{\beta}\), \(i = 1, 2, \dots, m\),   $\lambda\geq\overline{\lambda}$ given in Lemma \ref{thmm5.2sec}, and \(\phi_k^T \in L^2_{\mathcal{F}_T}(\Omega; L^2(G))\) (\(k = 1, 2\)), the corresponding solution of the system \eqref{ADJSO1sec} satisfies that
\begin{align}\label{ineq5.1car5.1sec}
\begin{aligned}
&\,\mathbb{E} \big|\phi_1(0)\big|^2_{L^2(G)} + \mathbb{E} \big|\phi_2(0)\big|^2_{L^2(G)} + \overline{I}_{0,T}(5,\phi_1) + \overline{I}_{0,T}(5,\phi_2)\\
&\quad + \overline{I}_{0,T}(3,h_1) + \overline{I}_{0,T}(3,h_2) \\
&\leq C \bigg[ \lambda^{29} \mathbb{E} \iint_{Q_0} \theta^2 \gamma^{29} |\phi_1|^2 \,dx\,dt + \lambda^{15} \mathbb{E} \iint_Q \theta^2 \gamma^{15} |\Phi_1|^2 \,dx\,dt \\
&\hspace{1cm}   + \lambda^{9} \mathbb{E} \iint_Q \theta^2 \gamma^{9} |\Phi_2|^2 \,dx\,dt \bigg],
\end{aligned}
\end{align}
where \(h_k = \displaystyle\sum_{i=1}^m \alpha_i \psi_k^i\), \; \(k = 1,2\).
\end{thm}

\section{Main controllability results}\label{section4} 
This section is devoted to proving our main null controllability results stated in Theorems \ref{th4.1SN} and \ref{th4.1SNsec}.
\subsection{Observability inequalities}
Here, we first prove the following observability inequality for the system \eqref{ADJSO1}.
\begin{prop}\label{Pro4.2}
Assume that \eqref{Assump10} and \eqref{assump1.3} hold. For  $\lambda=\overline{\lambda}$ given in Theorem \ref{lem4.5st}, there exists a constant \( C > 0 \), such that, for any \( \phi_k^T \in L^2_{\mathcal{F}_T}(\Omega; L^2(G)) \), \( k = 1,2 \), the associated solution of \eqref{ADJSO1} satisfies
\begin{align}\label{observaineq}
\begin{aligned}
&\,\mathbb{E}|\phi_1(0)|^2_{L^2(G)} + \mathbb{E}|\phi_2(0)|^2_{L^2(G)} + \sum_{i=1}^m \sum_{j=1}^2 \mathbb{E} \iint_Q  |\psi^i_j|^2 \,dx\,dt \\
&\leq C \mathbb{E} \iint_{Q} (\chi_{G_0} |\phi_1|^2 + |\Phi_1|^2 + |\Phi_2|^2) \,dx\,dt.
\end{aligned}
\end{align}
\end{prop}
\begin{proof}
Using Itô's formula for \( d(|\psi^i_1|^2+|\psi^i_2|^2) \), \( i = 1, 2, \dots, m \), we get for any \( t \in (0, T) \),
\begin{align}\label{inequa4.2}
\begin{aligned}
\mathbb{E} \int_0^t \int_G d(|\psi^i_1|^2+|\psi^i_2|^2) \, dx =& - 2 \mathbb{E} \int_0^t \int_G |\nabla\psi_1^i|^2 \, dx \, ds- 2 \mathbb{E} \int_0^t \int_G |\nabla\psi_2^i|^2 \, dx \, ds \\
&+ 2 \mathbb{E} \int_0^t \int_G a_{11} |\psi^i_1|^2 \, dx \, ds +2\mathbb{E} \int_0^t \int_G  a_{12}\psi^i_1\psi^i_2 \, dx \, ds  \\
&+ \frac{2}{\beta_i} \mathbb{E} \int_0^t \int_G \rho^{-2}_* \psi^i_1 \phi_1 \chi_{G_i} \, dx \, ds+2\mathbb{E} \int_0^t \int_G  a_{21}\psi^i_1\psi^i_2 \, dx \, ds \\
&+ 2 \mathbb{E} \int_0^t \int_G a_{22} |\psi^i_2|^2 \, dx \, ds.
\end{aligned}
\end{align}
By applying Young's inequality to the right-hand side of \eqref{inequa4.2}, we obtain that
\begin{align*}
\mathbb{E} \int_0^t \int_G d(|\psi^i_1|^2+|\psi^i_2|^2) \, dx &\leq C \left[ \mathbb{E} \iint_Q \rho^{-4}_* |\phi_1|^2 \,dx\,dt+\mathbb{E} \int_0^t \int_G (|\psi^i_1|^2+|\psi^i_2|^2) \, dx \, ds  \right], \quad i = 1, 2, \dots, m.
\end{align*}
Recalling that \( \psi^i_1(0)=\psi^i_2(0) = 0 \), for \( i = 1, 2, \dots, m \), and using Gronwall’s inequality, it follows that
\begin{align*}
\mathbb{E} \iint_Q  (|\psi^i_1|^2+|\psi^i_2|^2) \,dx\,dt \leq C \mathbb{E} \iint_Q \rho^{-4}_* |\phi_1|^2 \,dx\,dt, \quad i = 1, 2, \dots, m,
\end{align*}
which leads to
\begin{align}\label{estimm4.5}
\sum_{i=1}^m \sum_{j=1}^2 \mathbb{E} \iint_Q |\psi^i_j|^2 \,dx\,dt \leq C \mathbb{E} \iint_Q \overline{\theta}^2 |\phi_1|^2 \,dx\,dt.
\end{align}
This implies that
\begin{align}\label{estimm4.5secn}
 \sum_{i=1}^m \sum_{j=1}^2\mathbb{E} \iint_Q  |\psi^i_j|^2 \,dx\,dt \leq C \overline{I}_{0,T}(5, \phi_1).
\end{align}
Finally, combining \eqref{estimm4.5secn} and  \eqref{improvedCarl}, we deduce the desired observability inequality \eqref{observaineq}.
\end{proof}

We now prove the following observability inequality for the system \eqref{ADJSO1sec}.
\begin{prop}\label{Pro4.2sec}
Assume that \eqref{Assump10} and \eqref{assump1.3} hold. For  $\lambda=\overline{\lambda}$ given in Theorem \ref{lth5.1sec}, there exists a constant \(C > 0\) and a positive weight function \(\mathscr{\rho} = \rho(t)\) blowing up as \(t \rightarrow T\), such that, for any \(\phi_k^T \in L^2_{\mathcal{F}_T}(\Omega; L^2(G))\), \(k=1,2\), the associated solution of \eqref{ADJSO1sec} satisfies that
\begin{align}\label{observaineqsec}
\begin{aligned}
&\,\mathbb{E}|\phi_1(0)|^2_{L^2(G)}+\mathbb{E}|\phi_2(0)|^2_{L^2(G)}+ \sum_{i=1}^m \sum_{j=1}^2 \mathbb{E}\iint_Q \rho^{-2}|\psi^i_j|^2 \, dx\, dt \\
&\leq C \mathbb{E} \iint_{Q} (\chi_{G_0} |\phi_1|^2 + |\Phi_1|^2 + |\Phi_2|^2) \,dx\,dt.
\end{aligned}
\end{align}
\end{prop}
\begin{proof}
Define the function \(\rho(t) := e^{-\overline{\lambda} \overline{\alpha}^*(t)}\) where \(\overline{\alpha}^*(t) = \min_{x \in \overline{G}} \overline{\alpha}(t, x)\).
By Itô's formula, we compute \(d(\rho^{-2} |\psi^i_1|^2 + \rho^{-2} |\psi^i_2|^2)\), \(i = 1, 2, \dots, m\), then we conclude that for any \( t \in (0, T) \),
\begin{align}\label{inequa4.2sec}
\begin{aligned}
&\mathbb{E} \int_0^t \int_G d(\rho^{-2} |\psi^i_1|^2 + \rho^{-2} |\psi^i_2|^2) \, dx \\
&= -2\mathbb{E} \int_0^t \int_G \rho' \rho^{-3} (|\psi^i_1|^2 + |\psi^i_2|^2) \, dx \, ds - 2 \mathbb{E} \int_0^t \int_G \rho^{-2} (|\nabla\psi_1^i|^2 + |\nabla\psi_2^i|^2) \, dx \, ds \\
&\quad + 2 \mathbb{E} \int_0^t \int_G \rho^{-2} a_{11} |\psi^i_1|^2 \, dx \, ds + 2 \mathbb{E} \int_0^t \int_G \rho^{-2} a_{12} \psi^i_1 \psi^i_2 \, dx \, ds \\
&\quad + \frac{2}{\beta_i} \mathbb{E} \int_0^t \int_G \rho^{-2} \psi^i_1 \phi_1 \chi_{G_i} \, dx \, ds + 2 \mathbb{E} \int_0^t \int_G \rho^{-2} a_{21} \psi^i_1 \psi^i_2 \, dx \, ds \\
&\quad + 2 \mathbb{E} \int_0^t \int_G \rho^{-2} a_{22} |\psi^i_2|^2 \, dx \, ds.
\end{aligned}
\end{align}
Note that \( |\rho' \rho^{-3}| \leq C \rho^{-2} \), and applying Young's inequality to the right-hand side of \eqref{inequa4.2sec}, we obtain that for any \( i = 1, 2, \dots, m \),
\begin{align*}
\mathbb{E} \int_0^t \int_G d(\rho^{-2} |\psi^i_1|^2 + \rho^{-2} |\psi^i_1|^2) \, dx &\leq C\left[\mathbb{E} \iint_Q \rho^{-2} |\phi_1|^2 \,dx\,dt + \mathbb{E} \iint_Q \rho^{-2} (|\psi^i_1|^2 + |\psi^i_2|^2) \,dx\,dt \right].
\end{align*}
By Gronwall’s inequality, we conclude that
\begin{align}\label{ineqwithrhosec}
\mathbb{E} \iint_Q \rho^{-2} (|\psi^i_1|^2 + |\psi^i_2|^2) \,dx\,dt \leq C \mathbb{E} \iint_Q \rho^{-2} |\phi_1|^2 \,dx\,dt, \quad i = 1, 2, \dots, m,
\end{align}
which implies that
\begin{align}\label{estimm4.5sec}
\sum_{i=1}^m \sum_{j=1}^2 \mathbb{E} \iint_Q \rho^{-2} |\psi^i_j|^2 \,dx\,dt \leq C \mathbb{E} \iint_Q \rho^{-2} |\phi_1|^2 \,dx\,dt \leq C \overline{I}_{0,T}(5, \phi_1).
\end{align}
Combining \eqref{estimm4.5sec} and  \eqref{ineq5.1car5.1sec}, we deduce the observability inequality \eqref{observaineqsec}.
\end{proof}

\subsection{Proof of Theorems \ref{th4.1SN} and \ref{th4.1SNsec}}

Let us first establish the following null controllability result for the system \eqref{eqq4.7}, and then deduce the proof of Theorem \ref{th4.1SN}.
\begin{prop}\label{Lm5.5}
Assume that \eqref{Assump10} and \eqref{assump1.3} hold. Then, for any target functions \( (y^i_{1,d}, y^i_{2,d}) \in \mathscr{H}_{i,d} \) (\( i = 1, 2, \dots, m \)), and  any initial states \(y^0_1, y^0_2 \in L^2_{\mathcal{F}_0}(\Omega; L^2(G))\), there exist controls 
\((\widehat{u}_1, \widehat{u}_2, \widehat{u}_3) \in \mathscr{U}\), minimizing the functional \(J\), such that the corresponding solution of \eqref{eqq4.7} satisfies
\[
\widehat{y}_1(T,\cdot) = \widehat{y}_2(T,\cdot) = 0 \quad \textnormal{in} \; G, \quad \mathbb{P}\textnormal{-a.s.}
\]
Furthermore, there exists a constant $C>0$ such that
\begin{align}\label{eq54NScont2}
\begin{aligned}
&\, |\widehat{u}_1|^2_{L^2_\mathcal{F}(0,T; L^2(G_0))} + |\widehat{u}_2|^2_{L^2_\mathcal{F}(0,T; L^2(G))} + |\widehat{u}_3|^2_{L^2_\mathcal{F}(0,T; L^2(G))} \\
& \leq C \bigg( \mathbb{E} |y^0_1|^2_{L^2(G)} + \mathbb{E} |y^0_2|^2_{L^2(G)} + \sum_{i=1}^m \sum_{j=1}^2 \mathbb{E} \iint_{(0,T) \times \mathcal{O}_d}  |y^i_{j,d}|^2 \,dx\,dt \bigg).
\end{aligned}
\end{align}
\end{prop}

\begin{proof}
Applying Itô's formula for the systems \eqref{eqq4.7} and \eqref{ADJSO1}, we obtain the following duality relation
\begin{align}\label{ine6.6sec}
\begin{aligned}
&\,\mathbb{E}\int_G \left[y_1(T)\phi_1^T +y_2(T)\phi_2^T\right] dx -
\mathbb{E}\int_G \left[y^0_1\phi_1(0)+y^0_2\phi_2(0)\right] dx \\
&=\mathbb{E}\iint_{Q} (\chi_{G_0}u_1\phi_1+u_2\Phi_1+u_3\Phi_2) \, dx \,dt+\sum_{i=1}^{m}\sum_{j=1}^{2}\alpha_i \mathbb{E}\iint_Q y^i_{j,d} \psi^i_j \chi_{\mathcal{O}_d} \, dx \,dt.
\end{aligned}
\end{align}
Let \( \varepsilon > 0 \) and \( (\phi^T_1, \phi^T_2) \in L^2_{\mathcal{F}_T}(\Omega; L^2(G;\mathbb{R}^2)) \) be the terminal state of \eqref{ADJSO1}. We now define the following functional
\begin{align}\label{functJ}
\begin{aligned}
J_{\varepsilon}(\phi^T_1,\phi^T_2) &= \frac{1}{2} \mathbb{E} \iint_{Q} \left( \chi_{G_0} |\phi_1|^2 + |\Phi_1|^2 + |\Phi_2|^2 \right) dx \, dt + \varepsilon |(\phi^T_1, \phi^T_2)|_{L^2_{\mathcal{F}_T}(\Omega; L^2(G;\mathbb{R}^2))} \\
&\quad + \mathbb{E} \int_G \left[ y^0_1 \phi_1(0) + y^0_2 \phi_2(0) \right] dx + \sum_{i=1}^{m} \sum_{j=1}^2 \alpha_i \mathbb{E} \iint_Q y^i_{j,d} \psi^i_j \chi_{\mathcal{O}_d} \, dx \,dt.
\end{aligned}
\end{align}
It is straightforward to verify that \( J_{\varepsilon}: L^2_{\mathcal{F}_T}(\Omega; L^2(G;\mathbb{R}^2)) \to \mathbb{R} \) is continuous and strictly convex. Furthermore, by the observability inequality \eqref{observaineq}, the functional \( J_\varepsilon \) is coercive. Indeed, by applying Young's inequality, we have that for any \( \delta > 0 \), 
\begin{align*}
\mathbb{E} \int_G \left[ y^0_1 \phi_1(0) + y^0_2 \phi_2(0) \right] dx &\geq - \frac{\delta}{2} \left( \mathbb{E} |\phi_1(0)|^2_{L^2(G)} + \mathbb{E} |\phi_2(0)|^2_{L^2(G)} \right) - \frac{1}{2\delta} \left( \mathbb{E} |y^0_1|^2_{L^2(G)} + \mathbb{E} |y^0_2|^2_{L^2(G)} \right).
\end{align*}
Using the observability inequality \eqref{observaineq}, we deduce that
\begin{align}\label{inneq1}
\begin{aligned}
\mathbb{E} \int_G \left[ y^0_1 \phi_1(0) + y^0_2 \phi_2(0) \right] dx &\geq - \frac{\delta}{2} C \mathbb{E} \iint_{Q} \left( \chi_{G_0} |\phi_1|^2 + |\Phi_1|^2 + |\Phi_2|^2 \right) \,dx\,dt \\
&\quad + \frac{\delta}{2} \sum_{i=1}^m \sum_{j=1}^2 \mathbb{E} \iint_Q |\psi^i_j|^2 \,dx\,dt \\
&\quad - \frac{1}{2\delta} \left( \mathbb{E} |y^0_1|^2_{L^2(G)} + \mathbb{E} |y^0_2|^2_{L^2(G)} \right).
\end{aligned}
\end{align}
Similarly, we have the following estimate
\begin{align}\label{ineann2}
\begin{aligned}
\sum_{i=1}^{m} \sum_{j=1}^2 \alpha_i \mathbb{E} \iint_Q y^i_{j,d} \psi^i_j \chi_{\mathcal{O}_d} \,dx\,dt &\geq - \frac{\delta}{2} \sum_{i=1}^{m} \sum_{j=1}^2 \mathbb{E} \iint_Q  |\psi^i_j|^2 \,dx\,dt \\
&\quad - \frac{1}{2\delta} \sum_{i=1}^{m} \sum_{j=1}^2 \alpha_i^2 \mathbb{E} \iint_{(0,T) \times \mathcal{O}_d}  |y^i_{j,d}|^2 \,dx\,dt.
\end{aligned}
\end{align}
Combining \eqref{functJ}, \eqref{inneq1}, and \eqref{ineann2}, and choosing \( \delta = \frac{1}{2C} \), where \( C \) is the constant appearing in \eqref{inneq1}, we obtain that
\begin{align*}
J_{\varepsilon}(\phi^T_1, \phi^T_2) \geq &\, \varepsilon |(\phi^T_1, \phi^T_2)|_{L^2_{\mathcal{F}_T}(\Omega; L^2(G;\mathbb{R}^2))} - C \bigg( \mathbb{E} |y^0_1|^2_{L^2(G)} + \mathbb{E} |y^0_2|^2_{L^2(G)} \\
&+ \sum_{i=1}^m \sum_{j=1}^2 \alpha_i^2 \mathbb{E} \iint_{(0,T) \times \mathcal{O}_d}  |y^i_{j,d}|^2 \,dx\,dt \bigg).
\end{align*}
Hence, the functional \( J_{\varepsilon} \) is coercive in \( L^2_{\mathcal{F}_T}(\Omega; L^2(G; \mathbb{R}^2)) \). Consequently, \( J_{\varepsilon} \) admits a unique minimum \( (\phi^T_{1,\varepsilon}, \phi^T_{2,\varepsilon}) \). If \( (\phi^T_{1,\varepsilon}, \phi^T_{2,\varepsilon}) \neq (0,0) \), the first-order condition implies that
\begin{equation}\label{EulerLagraeq}
\left\langle J_\varepsilon'(\phi^T_{1,\varepsilon}, \phi^T_{2,\varepsilon}), (\phi^T_1, \phi^T_2) \right\rangle_{L^2_{\mathcal{F}_T}(\Omega; L^2(G;\mathbb{R}^2))} = 0, \quad \forall (\phi^T_1, \phi^T_2) \in L^2_{\mathcal{F}_T}(\Omega; L^2(G;\mathbb{R}^2)).
\end{equation}
Computing the derivative of \( J_\varepsilon \), the condition \eqref{EulerLagraeq} provides that
\begin{align}\label{Eq51NS}
\begin{aligned}
&\, \mathbb{E} \iint_{Q} (\chi_{G_0} \phi^{\varepsilon}_1 \phi_1+\Phi^{\varepsilon}_1 \Phi_1+\Phi^{\varepsilon}_2 \Phi_2) \,dx\,dt+ \varepsilon \frac{\left\langle (\phi^T_{1,\varepsilon}, \phi^T_{2,\varepsilon}), (\phi^T_1, \phi^T_2) \right\rangle_{L^2_{\mathcal{F}_T}(\Omega; L^2(G;\mathbb{R}^2))}}{|(\phi^T_{1,\varepsilon}, \phi^T_{2,\varepsilon})|_{L^2_{\mathcal{F}_T}(\Omega; L^2(G;\mathbb{R}^2))}} \\
& + \mathbb{E} \int_G \left[ y^0_1 \phi_1(0) + y^0_2 \phi_2(0) \right] dx + \sum_{i=1}^{m} \sum_{j=1}^2 \alpha_i \mathbb{E} \iint_Q y^i_{j,d} \psi^i_j \chi_{\mathcal{O}_d} \,dx\,dt = 0,
\end{aligned}
\end{align}
where $((\phi^\varepsilon_1,\phi^\varepsilon_2;\Phi^\varepsilon_1,\Phi^\varepsilon_2);\psi^{i,\varepsilon}_1,\psi^{i,\varepsilon}_2)$ is the solution of \eqref{ADJSO1} corresponding to the terminal state $(\phi^T_{1,\varepsilon}, \phi^T_{2,\varepsilon})$.\\
Choosing the controls \( (u^{\varepsilon}_1, u^{\varepsilon}_2, u^{\varepsilon}_3) = (\chi_{G_0} \phi^{\varepsilon}_1, \Phi^{\varepsilon}_1, \Phi^{\varepsilon}_2) \) in the identity \eqref{ine6.6sec} and combining the resulting equality with \eqref{Eq51NS}, we obtain the following equation for all \( (\phi^T_1, \phi^T_2) \in L^2_{\mathcal{F}_T}(\Omega; L^2(G;\mathbb{R}^2)) \):
\begin{equation*}
\varepsilon \frac{\left\langle (\phi^T_{1,\varepsilon}, \phi^T_{2,\varepsilon}), (\phi^T_1, \phi^T_2) \right\rangle_{L^2_{\mathcal{F}_T}(\Omega; L^2(G;\mathbb{R}^2))}}{|(\phi^T_{1,\varepsilon}, \phi^T_{2,\varepsilon})|_{L^2_{\mathcal{F}_T}(\Omega; L^2(G;\mathbb{R}^2))}} + \mathbb{E} \iint_Q \left[ y_{1}^\varepsilon(T) \phi^T_1 + y_{2}^\varepsilon(T) \phi^T_2 \right]\,dx\,dt = 0,
\end{equation*}
which immediately implies that
\begin{equation}\label{eq53NS}
\left| (y_{1}^\varepsilon(T), y_{2}^\varepsilon(T)) \right|_{L^2_{\mathcal{F}_T}(\Omega; L^2(G;\mathbb{R}^2))} \leq \varepsilon,
\end{equation}
where $(y_1^\varepsilon,y_2^\varepsilon;(z^{i,\varepsilon}_1,z^{i,\varepsilon}_2;Z^{i,\varepsilon}_1,Z^{i,\varepsilon}_2))$ is the solution of \eqref{eqq4.7} corresponding to the controls $(u^{\varepsilon}_1, u^{\varepsilon}_2, u^{\varepsilon}_3)$.\\
By taking the terminal state \( (\phi^T_{1}, \phi^T_{2}) = (\phi^T_{1,\varepsilon}, \phi^T_{2,\varepsilon}) \) in \eqref{Eq51NS} and using the observability inequality \eqref{observaineq} together with Young's inequality, we deduce that
\begin{align}\label{eq54NS}
\begin{aligned}
&\, |u^{\varepsilon}_1|^2_{L^2_\mathcal{F}(0,T; L^2(G_0))} + |u^{\varepsilon}_2|^2_{L^2_\mathcal{F}(0,T; L^2(G))} + |u^{\varepsilon}_3|^2_{L^2_\mathcal{F}(0,T; L^2(G))} \\
& \leq C \bigg( \mathbb{E} |y^0_1|^2_{L^2(G)} + \mathbb{E} |y^0_2|^2_{L^2(G)} + \sum_{i=1}^m \sum_{j=1}^2 \mathbb{E} \iint_{(0,T) \times \mathcal{O}_d}  |y^i_{j,d}|^2 \,dx\,dt \bigg).
\end{aligned}
\end{align}
If \( (\phi^T_{1,\varepsilon}, \phi^T_{2,\varepsilon}) = (0, 0) \), we have (see, e.g.,  \cite{FabPuZuaz95})
\begin{align}\label{Ezu.1}
\lim_{t \to 0^{+}} \frac{J_{\varepsilon}(t(\phi^T_1, \phi^T_2))}{t} \geq 0, \quad \forall (\phi^T_1, \phi^T_2) \in L^2_{\mathcal{F}_T}(\Omega; L^2(G;\mathbb{R}^2)).
\end{align}
By \eqref{Ezu.1}, and choosing \( (u_1^{\varepsilon}, u_2^{\varepsilon}, u_3^{\varepsilon}) = (0, 0, 0) \), we observe that \eqref{eq53NS} and \eqref{eq54NS} still hold. From \eqref{eq54NS}, we deduce the existence of a subsequence (still denoted by the same symbols) of \( (u_1^{\varepsilon}, u_2^{\varepsilon}, u_3^{\varepsilon}) \) such that, as \( \varepsilon \to 0 \), we have
\begin{align}\label{weakconvr}
\begin{aligned}
u^{\varepsilon}_1 &\longrightarrow  \widehat{u}_1 \quad \text{weakly in} \,\, L^2((0,T) \times \Omega; L^2(G_0)), \\
u^{\varepsilon}_2 &\longrightarrow  \widehat{u}_2 \quad \text{weakly in} \,\, L^2((0,T) \times \Omega; L^2(G)), \\
u^{\varepsilon}_3 &\longrightarrow  \widehat{u}_3 \quad \text{weakly in} \,\, L^2((0,T) \times \Omega; L^2(G)).
\end{aligned}
\end{align}
Thus, by \eqref{weakconvr}, we also have that
\begin{equation}\label{Eq56}
(y_{1}^\varepsilon(T), y_{2}^\varepsilon(T)) \longrightarrow (\widehat{y}_1(T), \widehat{y}_2(T)) \quad \text{weakly in} \; L^2_{\mathcal{F}_T}(\Omega; L^2(G;\mathbb{R}^2)), \quad \text{as} \, \varepsilon \to 0,
\end{equation}
where \( (\widehat{y}_1, \widehat{y}_2, (\widehat{z}^i_1, \widehat{z}^i_2; \widehat{Z}^i_1, \widehat{Z}^i_2)) \) is the solution of the system \eqref{eqq4.7} associated with the controls \( (\widehat{u}_1, \widehat{u}_2, \widehat{u}_3) \).  Finally, combining \eqref{eq53NS} and \eqref{Eq56}, we conclude that
\[
\widehat{y}_1(T,\cdot) = \widehat{y}_2(T,\cdot) = 0 \quad \text{in} \; G, \quad \mathbb{P}\textnormal{-a.s.}
\]
Subsequently, it is easy to see that the estimate \eqref{eq54NScont2} can be easily derived from \eqref{eq54NS} and \eqref{weakconvr}. This completes the proof of Proposition \ref{Lm5.5} and, consequently, proves the main Stackelberg-Nash null controllability result stated in Theorem \ref{th4.1SN}.
\end{proof}

Similarly to the proof of Proposition \ref{Lm5.5}, we also  establish the following null controllability result for the system \eqref{eqq4.7sec} and then deduce the proof of Theorem \ref{th4.1SNsec}.
\begin{prop}
Assume that \eqref{Assump10} and \eqref{assump1.3} hold. Let  \( \rho=\rho(t) \) be the weight function given in Proposition \ref{Pro4.2sec}. Then, for any target functions \( (y^i_{1,d}, y^i_{2,d}) \in \mathscr{H}_{i,d} \) (\( i = 1, 2, \dots, m \)) satisfying \eqref{asspfortargfst} and  any initial states \( y^0_1,y^0_2 \in L^2_{\mathcal{F}_0}(\Omega;L^2(G)) \), there exist controls 
\( (\widehat{u}_1, \widehat{u}_2,\widehat{u}_3) \in \mathscr{U} \), minimizing the functional \( J \) such that the corresponding solution of \eqref{eqq4.7sec} satisfies that
$$\widehat{y}_1(T,\cdot)=\widehat{y}_2(T,\cdot) =0 \quad\textnormal{in}\;G, \quad \mathbb{P}\textnormal{-a.s.}$$
Moreover, there exists a constant $C>0$ such that
\begin{align*}
\begin{aligned}
&\, |\widehat{u}_1|^2_{L^2_\mathcal{F}(0,T; L^2(G_0))} + |\widehat{u}_2|^2_{L^2_\mathcal{F}(0,T; L^2(G))} + |\widehat{u}_3|^2_{L^2_\mathcal{F}(0,T; L^2(G))} \\
& \leq C \bigg( \mathbb{E} |y^0_1|^2_{L^2(G)} + \mathbb{E} |y^0_2|^2_{L^2(G)} + \sum_{i=1}^m \sum_{j=1}^2 \mathbb{E} \iint_{(0,T) \times \mathcal{O}_d}  \rho^2|y^i_{j,d}|^2 \,dx\,dt \bigg).
\end{aligned}
\end{align*}
\end{prop}

\section{Concluding remarks}\label{secc5}
We have presented the Stackelberg-Nash strategy for coupled forward stochastic parabolic systems with one localized leader control, two additional leader controls in the noise terms, and \( m \) followers. Two scenarios for the followers are considered: The first scenario involves the followers minimizing a functional that depends on two components of the state, while the second scenario involves minimizing a functional that depends only on the second component. For both scenarios, and for fixed leader controls, the existence and uniqueness of the Nash equilibrium are established, along with its characterization. The problem is then reformulated as a classical null controllability issue for the associated coupled forward-backward stochastic parabolic system, which is obtained through suitable Carleman estimates. 

Now, we discuss an open question for future research: In \eqref{functji1}, each follower \( v_i \) (\( i = 1, 2, \dots, m \)) aims to minimize the functional \( J_i \), which involves both the states \( y_1 \) and \( y_2 \). The goal is to keep the state \( (y_1, y_2) \) close to the target functions \( (y^i_{1,d}, y^i_{2,d}) \), while using a control \( v_i \) with minimal energy, in the sense of the \( L^2_{\rho_*} \)-norm. On the other hand, in \eqref{functji1sec}, each follower \( v_i \) minimizes the functional \( \widetilde{J}_i \), which involves only the state \( y_2 \) and the \( L^2 \)-norm of \( v_i \). Following the same ideas presented in this paper, we can also consider the functionals \( (J^1_i)_{1 \leq i \leq m} \) defined by:
\[
J^1_i(u_1, u_2, u_3; v_1, \dots, v_m) = \frac{\alpha_i}{2} \mathbb{E} \iint_{(0,T) \times \mathcal{O}_{i,d}} |y_1 - y^i_{1,d}|^2 \,dx\,dt + \frac{\beta_i}{2} \mathbb{E} \iint_{(0,T) \times G_i} \rho_*^2(t) |v_i|^2 \,dx\,dt,
\]
where \( \rho_* \) is the same function defined in \eqref{functji1}. However, the Stackelberg-Nash controllability problem for the system \eqref{eqq1.1}, when the followers \( (v_i)_{1 \leq i \leq m} \) minimize the functionals \( (J^2_i)_{1 \leq i \leq m} \) or \( (J^3_i)_{1 \leq i \leq m} \) as defined below, remains an open problem:
\begin{align*}
J^2_i(u_1, u_2, u_3; v_1, \dots, v_m) = & \, \frac{\alpha_i}{2} \mathbb{E} \iint_{(0,T) \times \mathcal{O}_{i,d}} \left( |y_1 - y^i_{1,d}|^2 + |y_2 - y^i_{2,d}|^2 \right) \,dx\,dt \\
& + \frac{\beta_i}{2} \mathbb{E} \iint_{(0,T) \times G_i} |v_i|^2 \,dx\,dt,
\end{align*}
and
\[
J^3_i(u_1, u_2, u_3; v_1, \dots, v_m) = \frac{\alpha_i}{2} \mathbb{E} \iint_{(0,T) \times \mathcal{O}_{i,d}} |y_1 - y^i_{1,d}|^2 \,dx\,dt + \frac{\beta_i}{2} \mathbb{E} \iint_{(0,T) \times G_i} |v_i|^2 \,dx\,dt.
\]

\end{document}